\documentclass[12pt]{amsart}
\usepackage{amsmath}
\usepackage{amsthm}
\usepackage{amssymb}
\usepackage{amscd}
\usepackage{amsfonts}
\usepackage{amsbsy}
\usepackage{epsfig,afterpage}
\usepackage[dvips]{psfrag}
\usepackage{subfigure}
\usepackage{epsfig}
\usepackage{amsmath}
\usepackage{srcltx}
\usepackage{amsfonts}
\usepackage{amssymb}
\usepackage{psfrag}
\usepackage{verbatim}
\usepackage{graphicx}
\usepackage{amsmath}
\usepackage{amsthm}
\usepackage{amssymb}
\usepackage{amscd}
\usepackage{amsfonts}
\usepackage{amsbsy}
\usepackage{epsfig}
\usepackage[dvips]{psfrag}
\usepackage{multirow}
\usepackage{graphics}
\usepackage{epstopdf}
\usepackage{overpic}
\usepackage{float}

\usepackage{graphicx}

\newtheorem {theorem} {Theorem}
\newtheorem {proposition} [theorem]{Proposition}

\newtheorem {lemma}  [theorem]{Lemma}

\newtheorem {definition} [theorem]{Definition}

\usepackage{float}
\usepackage{tikz, wrapfig}
\tikzset{node distance=3cm, auto}
\allowdisplaybreaks

\begin{document}

\title[Resolution of singularities of constrained systems]
{Resolution of singularities of 2-dimensional real analytic constrained differential systems}

\author[O. H. Perez and P. R. da Silva]
{Otavio Henrique Perez and Paulo Ricardo da Silva}

\address{S\~{a}o Paulo State University (Unesp), Institute of Biosciences, Humanities and
	Exact Sciences. Rua C. Colombo, 2265, CEP 15054--000. S. J. Rio Preto, S\~ao Paulo,
	Brazil.}

\email{oh.perez@unesp.br}
\email{paulo.r.silva@unesp.br}

\thanks{ .}


\date{}
\dedicatory{}
\maketitle

\begin{abstract}
We present a theorem of resolution of singularities for real analytic constrained differential systems $A(x)\dot{x} = F(x)$ defined on a 2-manifold with corners having impasse set $\{x; \det A(x) = 0\}$. This result can be seen as a generalization of the classical one for 2-dimensional real analytic vector fields. Our proof is based on weighted blow-ups and the Newton polygon.
\end{abstract}

\section{Introduction and main result}

\noindent

A \textbf{$n$-dimensional constrained differential system} (or simply constrained system) in an open set $U\subseteq\mathbb{R}^{n}$ is given by
\begin{equation}\label{eq-def-constrained}
A(\mathbf{x})\dot{\mathbf{x}} = F(\mathbf{x}),
\end{equation}
where $\mathbf{x}\in U\subseteq\mathbb{R}^{n}$, $A:U\rightarrow Mat(\mathbb{R},n)$ and $F:U\rightarrow \mathbb{R}^{n}$ are smooth functions. A point $p$ such that $\det A(p) = 0$ is called \textbf{impasse point}, and the so called \textbf{impasse set} is the set $\Delta = \{p\in U; \det A(p) = 0\}$.

Along this paper, we suppose that $A$ does not have constant rank. We will further suppose that $\det A(\mathbf{x}) = 0$ only in a closed subset of $U$. This implies that the adjoint matrix $A^{*}$ of $A$ (the transpose of the cofactor matrix of $A$) is not identically zero and $AA^{*} = A^{*}A = \det(A)I$ (see Lemma 37.1, \cite{RabierRheinboldt}).

Observe that outside the impasse set a constrained system can be written as
$$\dot{\mathbf{x}} = A(\mathbf{x})^{-1}F(\mathbf{x}),$$
and we have a classical ordinary differential equation.

Multiplying both sides of \eqref{eq-def-constrained} by the adjoint matrix $A^{*}$ of $A$, it can be shown (see Lemma 37.2, \cite{RabierRheinboldt}) that $\gamma$ is solution of \eqref{eq-def-constrained} if, and only if, $\gamma$ is a solution of
\begin{equation}\label{eq-def-constrained-2}
\delta(\mathbf{x})\dot{\mathbf{x}} = A^{*}(\mathbf{x})F(\mathbf{x}),
\end{equation}
where $\delta: U\rightarrow \mathbb{R}$ is the determinant function and $\{\mathbf{x}\in U;\delta(\mathbf{x}) = 0\}$ is the impasse set $\Delta$. The vector field $X(\mathbf{x}) = A^{*}(\mathbf{x})F(\mathbf{x})$ will be called \textbf{adjoint vector field} and a constrained system in the form \eqref{eq-def-constrained-2} will be called \textbf{diagonalized constrained system}.

At an impasse point $p$, since the operator $A(p)$ is not invertible, the classical results on the existence and uniqueness of the solutions break down. However, near such point we can describe the phase portrait as follows. In the open set where $\det A(q) > 0$, the phase portrait is given by the adjoint vector field $X$. On the other hand, in the open set where $\det A(q) < 0$ the phase portrait is given by $-X$.

\begin{definition}\label{def-singularities}
We say that $p\in U$ is a \textbf{singularity of the constrained system} if one of the following conditions is satisfied:
\begin{enumerate}
  \item $p$ is an equilibrium point of the adjoint vector field $X$, that is, $X(p) = 0$;
  \item the impasse set $\Delta$ is not smooth at $p$;
  \item $p$ is a tangency point between $X$ and $\Delta$, that is, $d\delta (X) = 0$.
\end{enumerate}
Otherwise, we say that $p$ is \textbf{non-singular}. See Figures \ref{fig-def-sing} and \ref{fig-def-non-sing}.
\end{definition}

\begin{figure}[h]
  \center{\includegraphics[width=0.75\textwidth]{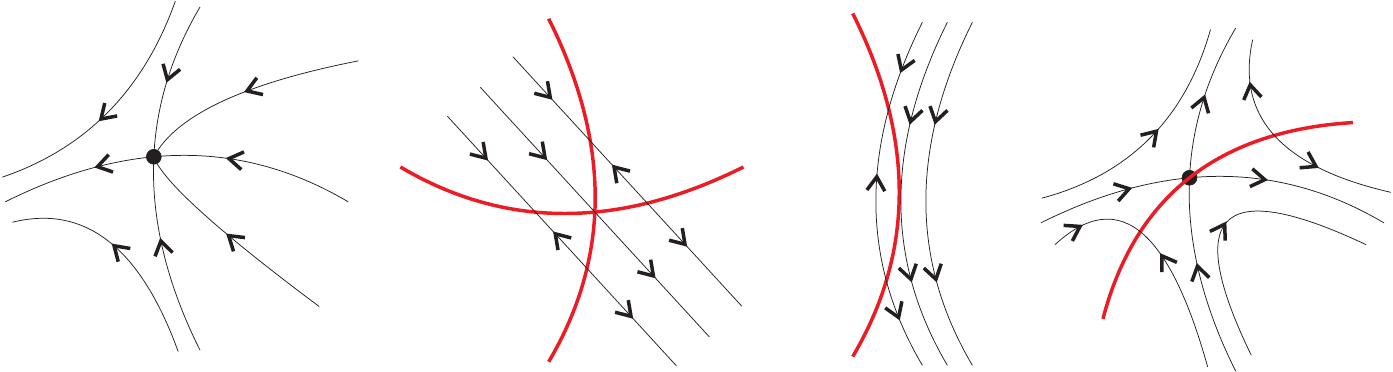}}\\
  \caption{Singularities of a planar constrained system. The impasse set is highlighted in red.}\label{fig-def-sing}
\end{figure}
\begin{figure}[h]
  \center{\includegraphics[width=0.5\textwidth]{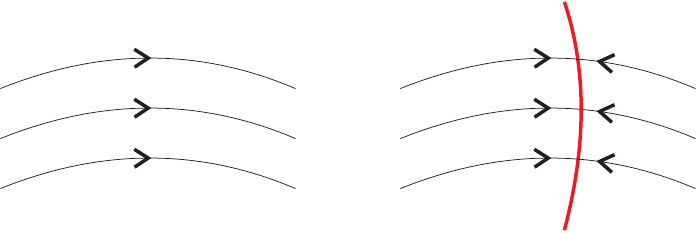}}\\
  \caption{Non-singular points of a planar constrained system. The impasse set is highlighted in red.}\label{fig-def-non-sing}
\end{figure}

From now on we consider 2-dimensional constrained differential systems. We say that an equilibrium point of a vector field is semi-hyperbolic if at least one of the eigenvalues of its linearization has nonzero real part.

\begin{definition}\label{def-desing2}
A 2-dimensional constrained differential system is \textbf{elementary at $p$} if one of the following conditions holds:
\begin{enumerate}
  \item $p$ is a non singular point;
  \item $p$ is a semi-hyperbolic equilibrium point of $X$ and $p\not\in\Delta$;
  \item If $p$ is a semi-hyperbolic equilibrium point of $X$ and $p\in\Delta$, then $\Delta$ coincides with a local separatrix of $X$ at $p$.
\end{enumerate}
See Figure \ref{fig-ementary-points}.
\end{definition}

\begin{figure}[h!]
  \center{\includegraphics[width=0.30\textwidth]{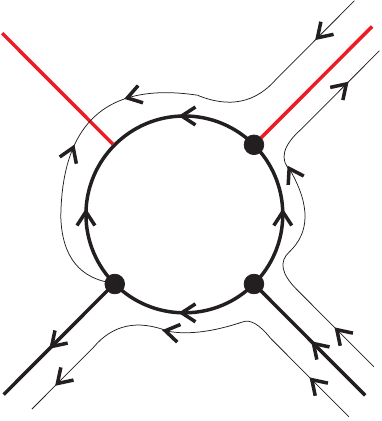}}\\
  \caption{Elementary points of a planar constrained system. The impasse set is highlighted in red.}\label{fig-ementary-points}
\end{figure}

In \cite{SotoZhito,Zhito} the authors studied singularities of constrained systems near a smooth point of the impasse set, in dimensions $n\geq 3$ and $n = 2$, respectively. Here we also consider singular points of the impasse set. Beyond the mathematical interest in such problem, some models from applied sciences may present this kind of singularity (e.g. see Example 1, \cite{Smale}). In \cite{CardinSilvaTeixeira} the authors considered homogeneous blow ups to study singular points of the impasse set of real polynomial constrained systems whose impasse set is defined by a homogeneous polynomial, among others non-degeneracy conditions. In the present work we are going to generalize these results by proving a theorem of resolution of singularities for real analytic planar constrained systems without such non-degeneracy assumptions.

In order to state our main theorem, we need to work in a more general category of analytic manifolds with corners. Roughly speaking, a 2-dimensional manifold with corners is a topological space locally modeled by open sets of $(\mathbb{R}_{\geq 0})^{2}$. We refer to \cite{Joyce} for a detailed discussion on this subject. As we will see in the discussion below, after a blow-up in a singularity we obtain a new ambient space, which is a manifold with corners. Throughout this paper, $\mathcal{M}$ is a 2-dimensional real analytic manifold with corners. See Figure \ref{fig-manifold-with-corners}.

\begin{figure}[h!]
  \center{\includegraphics[width=0.55\textwidth]{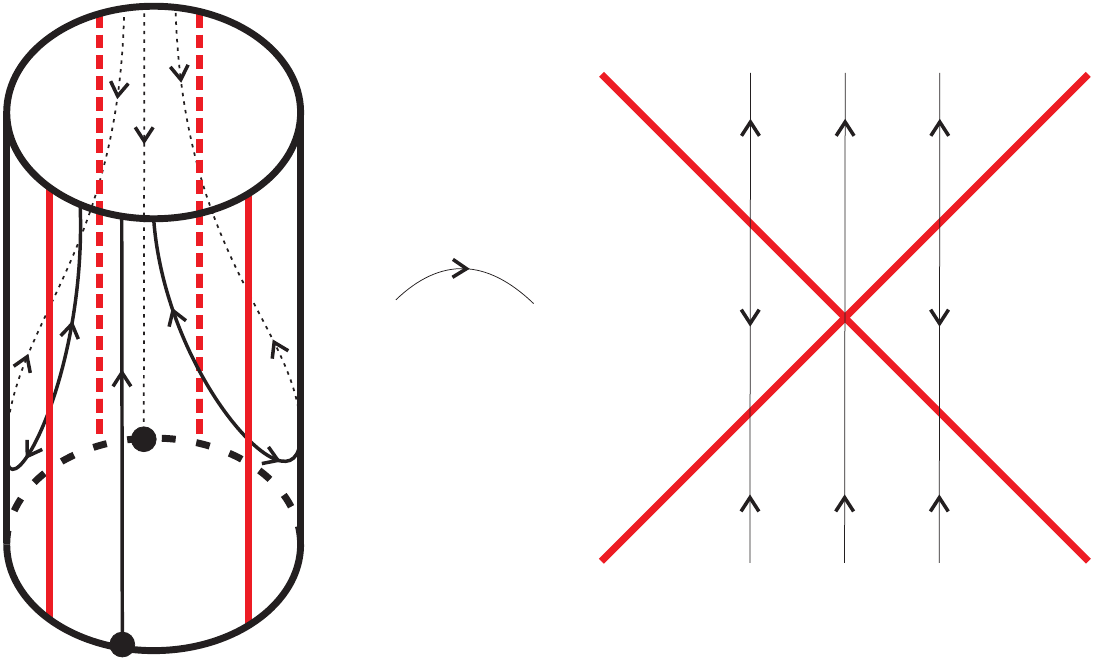}}\\
  \caption{Constrained differential system in a 2-manifold with corners after a blow-up.}\label{fig-manifold-with-corners}
\end{figure}

A \textbf{1-dimensional analytic oriented foliation} on $\mathcal{M}$ is a collection of pairs $\mathcal{X} = \{(U_{i}, X_{i})\}_{i\in I}$ such that $\{U_{i}\}_{i\in I}$ is an open covering of $\mathcal{M}$, $X_{i} = (P_{i},Q_{i})$ is a real analytic vector field defined in $U_{i}$ and, for each $i,j\in I$, there is a strictly positive function $u_{i,j}$ such that $X_{i} = u_{i,j}X_{j}$ on $U_{i}\cap U_{j}$. We further suppose that $\partial\mathcal{M}$ is invariant by the foliation and $P_{i}$ and $Q_{i}$ do not have common factor, which implies that equilibrium points of $X_{i}$ are isolated.

We define the \textbf{principal ideal sheaf} $\mathcal{I}$ as the collection of pairs $\mathcal{I} = \{(U_{i}, \delta_{i})\}_{i\in I}$ such that $\{U_{i}\}_{i\in I}$ is an open covering of $\mathcal{M}$, $\delta_{i}$ is an irreducible analytic function defined in $U_{i}$ and, for each $i,j\in I$, there is a strictly positive function $v_{i,j}$ such that $\delta_{i} = v_{i,j}\delta_{j}$ on $U_{i}\cap U_{j}$.

A constrained differential system defined in a 2-dimensional manifold with corners is a triple $(\mathcal{M}, \mathcal{X}, \mathcal{I})$, where $\mathcal{M}$ is a 2-dimensional real analytic manifold with corners, $\mathcal{X}$ is a 1-dimensional analytic oriented foliation on $\mathcal{M}$ and $\mathcal{I}$ is the principal ideal sheaf. On each $U_{i}$ of the open covering, we associate the diagonalized constrained differential system
$$\delta_{i}(\mathbf{x})\dot{\mathbf{x}} = X_{i}(\mathbf{x}).$$

For shortness we say that $\mathcal{I}$ is the \textbf{impasse ideal} of the constrained system and its zero set $\Delta$ (locally given by $\{\delta_i = 0\}$) is the impasse set. A straightforward computation shows that the phase curves of the constrained system coincide in the intersection $U_{i}\cap U_{j}$ of two open sets of the covering.

Singularities and elementary points are defined in the same way as Definitions \ref{def-singularities} and \ref{def-desing2}. We say that \textbf{$(\mathcal{M},\mathcal{X},\mathcal{I})$ is elementary} if $(\mathcal{M},\mathcal{X},\mathcal{I})$ is elementary at every $p\in \mathcal{M}$.

Now we are able to state the main theorem.

\begin{theorem}\label{teo-resolution-singularities}
Let $(\mathcal{M}, \mathcal{X}, \mathcal{I})$ be a 2-dimensional real analytic constrained differential system defined on a compact manifold with corners. Then there is a finite sequence of weighted blow ups
$$(\widetilde{\mathcal{M}}, \widetilde{\mathcal{X}}, \widetilde{\mathcal{I}}) = (\mathcal{M}_{n}, \mathcal{X}_{n}, \mathcal{I}_{n}) \xrightarrow{\Phi_{n}} \cdots \xrightarrow{\Phi_{0}} (\mathcal{M}_{0}, \mathcal{X}_{0}, \mathcal{I}_{0}) = (\mathcal{M}, \mathcal{X}, \mathcal{I})$$
such that $(\widetilde{M}, \widetilde{X}, \widetilde{\mathcal{I}})$ is elementary.
\end{theorem}

We will give the proof of this theorem in Section \ref{sec-global-resolution}.

The first theorem of resolution of singularities for 2-dimensional real analytic vector fields was stated by Bendixson \cite{Bendixson}, and its first proof was given by Seidenberg \cite{Seidenberg}. Such result was generalized by Dumortier \cite{Dumortier} for real $C^{\infty}$ planar vector fields. Brunella and Miari \cite{BrunellaMiari} considered weighted blow-ups to resolve isolated singularities of 2-dimensional real analytic vector fields under some non degeneracy conditions. Pelletier \cite{Pelletier} gave a proof dropping such non degeneracy conditions, and in the same paper the author also estimates the number of weighted blow-ups needed to resolve a singularity. The proof of resolution of singularities for 3-dimensional real analytic vector fields was given by Panazzolo \cite{Panazzolo}. Just as in \cite{BelottoPanazzolo, Panazzolo, Pelletier}, weighted blow-ups and the Newton polygon are our main tools.

Resolution of singularities is an important tool in the study of equilibrium points of real planar vector fields (see \cite{BrunellaMiari,Dumortier} for instance). Following this direction, the techniques presented here will be useful in the topological classification of phase portraits of constrained differential systems near a singularity. Indeed, such study is the next step once the theorem of resolution of singularities is proved (see \cite{PerezSilva}).

The paper is organized as follows. In Section \ref{sec-main-tools} we present some preliminary definitions necessary for the development of the study. In Section \ref{sec-properties} the Newton polygon is considered in the context of constrained systems and we find suitable coordinate systems that will be useful in the proof of the main theorem. In Section \ref{sec-local-strategy} we given a detailed description of the algorithm of resolution of singularities of planar constrained differential systems. Finally, in Section \ref{sec-global-resolution} we prove Theorem \ref{teo-resolution-singularities}.

\section{Main tools}\label{sec-main-tools}
\noindent

In this section we present the tools used in the resolution of singularities of planar constrained differential systems. The proof of the main theorem is based on weighted blow-ups and the weights are chosen by means of the Newton polygon. For our knowledge, it is not defined in the literature such a polygon for constrained systems. The idea is to associate an analytic vector field to the constrained differential system (the so called auxiliary vector field) and then we study the Newton polygon of such vector field. In the end of this section, an example is presented in order to explain all the tools defined.

Firstly, we recall the notion of orbital equivalence presented in \cite{SotoZhito}. Two constrained systems $\delta_{i}(\mathbf{x})\dot{\mathbf{x}} = X_{i}(\mathbf{x})$, $i = 1,2,$  are \textbf{orbitally equivalent at the points $p_{i}\in\Delta_{i}$}, $i = 1,2$, if there are two open sets $V_{1}\ni p_{1}$ and $V_{2}\ni p_{2}$ and an analytic diffeomorphism $\Phi: V_{1}\rightarrow V_{2}$ such that
\begin{enumerate}
  \item $\Phi$ maps the impasse set $\Delta_{1}$ to the impasse set $\Delta_{2}$;
  \item $\Phi$ maps the phase curves of $X_{1}$ in $V_{1}\backslash\Delta_{1}$ to the phase curves of $X_{2}$ in $V_{2}\backslash\Delta_{2}$, not necessarily preserving orientation.
\end{enumerate}
In the case where $\Phi$ preserves orientation, we say that $\Phi$ is an \textbf{orientation preserving orbital equivalence}.

\subsection{Quasi-homogeneous blow-up}

\noindent

In what follows we introduce weighted blow-ups (or quasi-homogeneous blow ups) following the definitions presented in \cite{Panazzolo} for blow-ups in $n$-dimensional manifolds. For a good introduction of the method for 2-dimensional vector fields, we refer to \cite{AlvarezFerragutJarque, DumortierLlibreArtes}. Roughly speaking, blowing-up a singularity means to replace a point by a compact set, which gives us a new ambient space. In this new compact set, new (possibly less degenerated) singularities may appear.

Let $(\mathcal{M}, \mathcal{X}, \mathcal{I})$ be a 2-dimensional real analytic constrained system and fix local coordinates of such triple. A \textbf{weight-vector} is a non-zero vector $\omega = (\omega_{1}, ..., \omega_{n})$ of positive integers. A \textbf{quadratic $\omega$-blow-up} is a real analytic surjective map

\begin{center}
\begin{tabular}{lccc}

  $\Phi_{\omega}:$ & $\widetilde{M} = \mathbb{S}^{n-1}\times\mathbb{R}_{\geq 0}$ & $\rightarrow$ & $\mathbb{R}^{n}$, \\
   & $(\bar{x},\tau)$ & $\mapsto$ & $\tau^{\omega}\bar{x} = (\tau^{\omega_{1}}\bar{x}_{1},...,\tau^{\omega_{n}}\bar{x}_{n})$, \\

\end{tabular}
\end{center}
where $\mathbb{S}^{n-1} = \{\bar{x}\in\mathbb{R}^{n};\sum_{i=1}^{n}\bar{x}_{i}^{2} = 1\}$. The set $\Omega = \{(0,...,0)\}$ is called \textbf{blow-up center} and $\Phi_{\omega}^{-1}(\Omega) = \mathbb{S}^{n-1}\times\{0\} = \mathbb{E}$ is called \textbf{exceptional divisor}.

Notice that $\Phi_{\omega}|_{\widetilde{M}\backslash \mathbb{E}}$ is an analytic diffeomorphism that maps its domain onto $\mathbb{R}^{n}\backslash\Omega$. If we blow up a point outside the boundary $\partial \mathcal{M}$, then $\mathbb{E}$ is the boundary of $\widetilde{\mathcal{M}}$. On the other hand, if we blow up a point on $\partial \mathcal{M}$, then we previously have a divisor $\mathcal{D}$ on $\mathcal{M}$ and the boundary of $\widetilde{\mathcal{M}}$ is given by $\widetilde{\mathcal{D}} = \Phi^{-1}_{\omega}(\mathcal{D})\cap\mathbb{E}$.

Given a 1-dimensional analytic oriented foliation $\mathcal{X}$ on $\mathcal{M}$, the \textbf{pull back of $\mathcal{X}$} is the 1-dimensional analytic oriented foliation $\widetilde{\mathcal{X}} = \Phi_{\omega}^{*}\mathcal{X}$ defined in $\widetilde{\mathcal{M}}$, which is orientation preserving orbitally equivalent to $\mathcal{X}$ outside $\mathbb{E}$. Analogously, the \textbf{pull back of $\mathcal{I}$} is the impasse ideal $\widetilde{\mathcal{I}} = \Phi_{\omega}^{*}\mathcal{I}$ defined in $\widetilde{\mathcal{M}}$ whose generators are diffeomorphic to that ones of $\mathcal{I}$ outside $\mathbb{E}$. The \textbf{strict transformed of $(\mathcal{M},\mathcal{X},\mathcal{I})$} as the triple $(\widetilde{\mathcal{M}},\widetilde{\mathcal{X}},\widetilde{\mathcal{I}})$. It can be easily checked that outside the exceptional divisor the triples $(\mathcal{M},\mathcal{X},\mathcal{I})$ and $(\widetilde{\mathcal{M}},\widetilde{\mathcal{X}},\widetilde{\mathcal{I}})$ are orbitally equivalent orientation preserving.

In many situations, it is more practical to consider directional charts instead the quadratic blow up. Take the weight vector $\omega = (\omega_{1}, ..., \omega_{n})$ and $\sigma\in\{+,-\}$. Consider the set $V_{i}^{\sigma}$
$$V_{i}^{\sigma} = \{(\bar{x},\tau)\in\mathbb{S}^{n-1}\times\mathbb{R}_{\geq0}; \sigma x_{i} > 0\},$$
take the projection

\begin{center}
\begin{tabular}{lccc}

  $\Pi_{i}^{\sigma}:$ & $V_{i}^{\sigma}$ & $\rightarrow$ & $\mathbb{R}^{n-1}\times\mathbb{R}_{\geq0}$, \\
   & $(\bar{x},\tau)$ & $\mapsto$ & $(\bar{x}_{1},...,\bar{x}_{i-1},\tau,\bar{x}_{i+1},...,\bar{x}_{n})$, \\

\end{tabular}
\end{center}
and define the map

\begin{center}
\begin{tabular}{lccc}

  $\Phi_{i,\omega}^{\sigma}:$ & $\mathbb{R}^{n-1}\times\mathbb{R}_{\geq0}$ & $\rightarrow$ & $\mathbb{R}^{n}$, \\
   & $(\tilde{x}_{1},...,\tilde{x}_{n})$ & $\mapsto$ & $(\tilde{x}_{i}^{\omega_{1}}\tilde{x}_{1}, ..., \tilde{x}_{i}^{\omega_{i-1}}\tilde{x}_{i-1},\sigma\tilde{x}_{i}^{\omega_{i}},\tilde{x}_{i}^{\omega_{i+1}}\tilde{x}_{i+1},...,\tilde{x}_{i}^{\omega_{n}}\tilde{x}_{n})$. \\

\end{tabular}
\end{center}

The \textbf{$x_{i}$-directional chart} is given by the pairs $(V_{i}^{\sigma},\Pi_{i}^{\sigma})$ and the \textbf{blow-up in the $x_{i}$ direction} is given by $\Phi_{i,\omega}^{\sigma}$. It can be shown that Diagram \ref{fig-def-r-eq} is commutative.

\begin{figure}[h!]
\begin{flushright}
\begin{center}
\begin{tikzpicture}
\node (A) {$V_{i}^{\sigma}$};
\node (B) [right of=A] {$\mathbb{R}^{n}\cap\{\sigma x_{i}\geq 0\}$};
\node (C) [below of=A] {$U$};
\node (D) [below of=B] {$\mathbb{R}^{n}\cap\{\sigma x_{i}\geq 0\}$};
\large\draw[->] (A) to node {\mbox{{\footnotesize $\Phi_{\omega}$}}} (B);
\large\draw[->] (A) to node {\mbox{{\footnotesize $\Pi^{\sigma}_{i}$}}} (C);
\large\draw[->] (B) to node {\mbox{{\footnotesize $id$}}} (D);
\large\draw[->] (C) to node {\mbox{{\footnotesize $\Phi^{\sigma}_{i,\omega}$}}} (D);
\end{tikzpicture}
\end{center}
\end{flushright}
\caption{Equivalence between directional and quadratic blow-ups.}\label{fig-def-r-eq}
\end{figure}

The pull back of $\mathcal{X}$, the pull back of $\mathcal{I}$ and the strict transformed $(\widetilde{\mathcal{M}},\widetilde{\mathcal{X}},\widetilde{\mathcal{I}})$ of $(\mathcal{M},\mathcal{X},\mathcal{I})$ are defined in the same way as before.

\subsection{The Newton polygon}
\noindent

A question that naturally rises is how we chose the weight vector of the weighted blow-up. In order to make a good choice for the resolution of singularities, we utilize the Newton polygon associated to a 2-dimensional analytic vector field. In this subsection we briefly recall the construction of such mathematical object, which can also be found in \cite{AlvarezFerragutJarque,DumortierLlibreArtes, Pelletier}.

We say that a polynomial $P(x,y) = \displaystyle\sum p_{r,s}x^{r}y^{s}$ is \textbf{quasi-homogeneous of degree $d$} if there is a vector of positive integers $\omega = (\omega_{1},\omega_{2})$ such that $P(t^{\omega_{1}}x,t^{\omega_{2}}y) = t^{d}P(x,y)$. We say that $d$ is the \textbf{degree of quasi-homogeneity}.

Let $X$ be an analytic vector field. Just as in \cite{Panazzolo}, we write $X$ in the so called logarithmic basis. More precisely, consider
$$X(x,y) = P(x,y)\displaystyle\frac{\partial}{\partial x} + Q(x,y)\displaystyle\frac{\partial}{\partial y},$$
where
$$P(x,y) = \Big{(}\sum a_{m,n}x^{m}y^{n}\Big{)}x, \ Q(x,y) = \Big{(}\sum b_{m,n}x^{m}y^{n}\Big{)}y,$$
with $a_{m,n},b_{m,n}\in\mathbb{R}$ and $m,n\in \mathbb{Z}$ satisfying:
\begin{enumerate}
  \item For $m < -1$ or $n \leq -1$, $a_{m,n} = 0$;
  \item For $m \leq -1$ or $n < -1$, $b_{m,n} = 0$.
\end{enumerate}

We recall that a diagonalized constrained system is written in the form
\begin{equation}\label{eq-impasse}
    \delta(x,y)\dot{x} = P(x,y), \ \delta(x,y)\dot{y} = Q(x,y),
\end{equation}
with $\delta (x,y)$ being an irreducible real-analytic function that can be written as
$$\delta(x,y) = \sum c_{k,l}x^{k}y^{l},$$
where $c_{k,l}\in\mathbb{Z}$ is such that $c_{k,l} = 0$ when $k < 0$ or $l<0$.

Let $\omega = (\omega_{1},\omega_{2})$ be a vector of positive integers. One can write the planar vector field $X$ as
\begin{equation}\label{def-graduation}
X(x,y) = \displaystyle\sum_{d = -1}^{\infty}X_{d}^{(\omega_{1},\omega_{2})}(x,y),
\end{equation}
where
\begin{equation}\label{def-d-level-graduation}
X_{d}^{(\omega_{1},\omega_{2})}(x,y) = \displaystyle\sum_{\omega_{1} r + \omega_{2} s = d} x^{r}y^{s}\Big{(}a_{r,s}x\displaystyle\frac{\partial}{\partial x} +  b_{r,s}y\displaystyle\frac{\partial}{\partial y}\Big{)}.
\end{equation}

In other words, we are writing the vector field $X$ as a sum of quasi-homogeneous components. We say that \eqref{def-graduation} is a \textbf{$(\omega_{1},\omega_{2})$-graduation of $X$}. Each $X_{d}^{(\omega_{1},\omega_{2})}$ is called \textbf{$d$-level of the $(\omega_{1},\omega_{2})$-graduation}.

Given a $(\omega_{1},\omega_{2})$-graduation of $X$, we associate the monomials $a_{r,s}x^{r}y^{s}$ and $b_{r,s}x^{r}y^{s}$ with nonzero coefficients to a point $(r,s)$ in the plane of powers. Observe that each point $(r,s)$ is contained in a line of the form $\{\omega_{1} r + \omega_{2} s = d\}$. Furthermore, notice that the map that relates a monomial with a point $(r,s)$ is not injective.

\begin{definition}
The \textbf{support $\mathcal{Q}$ of $X$} is the set
$$\mathcal{Q} = \{(r,s)\in \mathbb{Z}^{2}; a_{r,s}^{2} + b_{r,s}^{2} \neq 0\}.$$
\end{definition}

\begin{definition}
The \textbf{Newton polygon $\mathcal{P}$} associated to the analytic vector field $X$ is the convex envelope of the set $\mathcal{Q} + \mathbb{R}^{2}_{+}$.
\end{definition}

The boundary of the Newton polygon $\partial\mathcal{P}$ is the union of a finite number of segments. We enumerate them from the left to the right: $\gamma_{0}, \gamma_{1}, ..., \gamma_{n+1}$. Observe that $\gamma_{0}$ is vertical and $\gamma_{n+1}$ is horizontal. Analogously, the non-regular points of $\partial\mathcal{P}$ will be enumerated from the left to the right: $v_{0}, ..., v_{n}$.

\begin{definition}
We say that $v_{0}, ..., v_{n}$ are the \textbf{vertices of the Newton polygon}. The vertex $v_{0} = (r_{0}, s_{0})$ is called \textbf{main vertex} and the segment $\gamma_{1}$ will be called \textbf{main segment}. The number $s_{0}$ is called \textbf{height of the Newton polygon}. See Figure \ref{fig-def-newton-polygon}.
\end{definition}

\begin{figure}[h]
  \center{\includegraphics[width=0.30\textwidth]{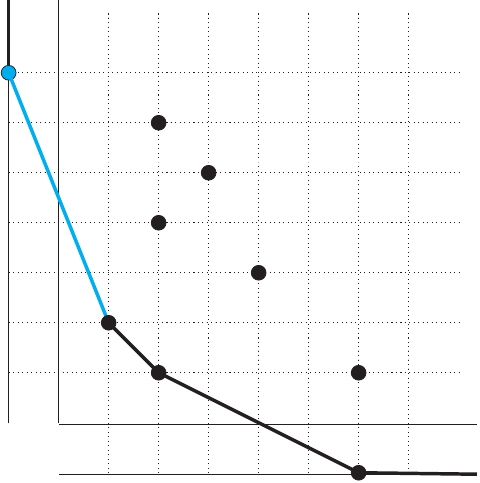}}\\
  \caption{Newton polygon of a vector field. The main vertex and main segment are highlighted in blue.}\label{fig-def-newton-polygon}
\end{figure}

The main segment $\gamma_{1}$ is contained in an affine line of the form $\{r\omega_{1} + s\omega_{2} = R\}$. Observe that the vector $\omega = (\omega_{1},\omega_{2})$ is normal with respect to $\{r\omega_{1} + s\omega_{2} = R\}$. In the resolution of singularities of analytic vector fields, one chooses the vector $\omega$ as the weight vector (see \cite{AlvarezFerragutJarque,Pelletier}).

The Newton polygon strongly depends on the coordinate system adopted. It means that applying a change of coordinates on a vector field, we possibly obtain a different Newton polygon and it may be an obstacle during the resolution process. Later on, we will see how to solve this problem.

\subsection{The auxiliary vector field}
\noindent

Consider the diagonalized constrained system given by \eqref{eq-impasse}, that is,
$$\delta(x,y)\dot{x} = P(x,y), \ \delta(x,y)\dot{y} = Q(x,y).$$

\begin{definition} The \textbf{auxiliary vector field} $X_{A}$ associated to the constrained differential equation \eqref{eq-impasse} is the vector field
\begin{equation}\label{def-aux-vec-field}
\scriptstyle X_{A}(x,y) = \delta(x,y)\Bigg{(}P(x,y)\frac{\partial}{\partial x} + Q(x,y)\frac{\partial}{\partial y}\Bigg{)} = \Bigg{(}\sum c_{k,l}x^{k}y^{l}\Bigg{)}\Bigg{(}\sum x^{m}y^{n}\big{(}a_{m,n}x\frac{\partial}{\partial x} + b_{m,n}y\frac{\partial}{\partial y}\big{)}\Bigg{)},
\end{equation}
which is real analytic.
\end{definition}

It is easy to sketch the phase portrait of the auxiliary vector field. Outside the impasse set, the auxiliary vector field \eqref{def-aux-vec-field} is obtained by multiplying the constrained differential system \eqref{eq-impasse} by the positive function $\delta^{2}$. On the other hand, the impasse set $\Delta$ is a curve of equilibrium points for \eqref{def-aux-vec-field}.

Since the auxiliary vector field is analytic, all previous definitions and remarks concerning the Newton polygon remain true for \eqref{def-aux-vec-field}. However, observe that the points on the plane of powers are of the form $(k+m, l + n)$, and therefore the support $\mathcal{Q}$ takes form
$$\mathcal{Q} = \{(k+m,l+s)\in \mathbb{Z}^{2}; c_{k,l}(a_{r,s}^{2} + b_{r,s}^{2}) \neq 0\}.$$

Moreover, the levels of a $(\omega_{1},\omega_{2})$-graduation are written in the form
$$X_{A,d}^{(\omega_{1},\omega_{2})}= \Bigg{(} \displaystyle\sum_{\omega_{1} k + \omega_{2} l = d_{1}} c_{k,l}x^{k}y^{l}  \Bigg{)} \Bigg{(}   \displaystyle\sum_{\omega_{1} m + \omega_{2}n = d_{2}} x^{m}y^{n}\Big{(}a_{m,n}x\displaystyle\frac{\partial}{\partial x} +  b_{m,n}y\displaystyle\frac{\partial}{\partial y}\Big{)}\Bigg{)},$$
where $d_{1} + d_{2} = d$. In other words, the support of the auxiliary vector field is obtained by the Minkowski Sum \cite{Schneider} between the points of the support of the adjoint vector field $X$ and the points of the support of $\delta$. Finally, we remark that the Newton polygon of the adjoint vector field and the auxiliary vector field are not necessarily the same.

For our knowledge, it does not exist a definition of Newton polygon for constrained systems. In order to solve this problem, we associate a planar constrained differential system to its auxiliary vector field $X_{A}$, and then we utilize the Newton polygon of $X_{A}$. Furthermore, the main advantage of writing the auxiliary vector this way is that it turns our computations easier, and it will be easy to see what is happening with the impasse set and the adjoint vector field during the resolution process.

\subsection{An example}
\noindent

In what follows we exemplify the concepts presented. Consider the following constrained differential equation
\begin{equation}\label{exe-cusp-fold}
xy\dot{x} = y, \ xy\dot{y} = x^{2},
\end{equation}
whose auxiliary vector field is given by
\begin{equation}\label{exe-cusp-fold-auxiliary}
X_{A}(x,y) = (xy)y\displaystyle\frac{\partial}{\partial x} + (xy)x^{2}\displaystyle\frac{\partial}{\partial y}.
\end{equation}

The support of \eqref{exe-cusp-fold-auxiliary} is the set $\mathcal{Q} = \{(0,2), (3,0)\}$. Therefore, the main vertex $v_{0}$ of the Newton polygon associated to \eqref{exe-cusp-fold-auxiliary} is $(0,2)$, the height of the polygon is equals $2$ and the main segment $\gamma_{1}$ is contained in the line $\{2r + 3s = 6\}$. See Figure \ref{fig-exe-cusp-fold}.

\begin{figure}[h]
  \center{\includegraphics[width=0.30\textwidth]{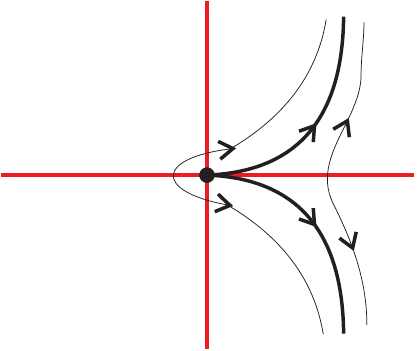}\hspace{1cm}\includegraphics[width=0.30\textwidth]{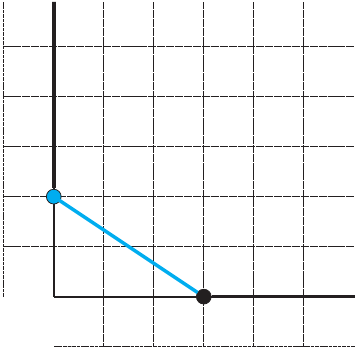}}\\
  \caption{Phase portrait of \eqref{exe-cusp-fold} (left) and Newton polygon of \eqref{exe-cusp-fold-auxiliary} (right). The impasse set is highlighted in red. The main vertex and the main segment of the polygon are highlighted in blue.}\label{fig-exe-cusp-fold}
\end{figure}

In the proof of the main theorem we consider directional blow ups. The weight vector will be that one that is normal to the affine line $\{2r + 3s = 6\}$. In this example, such weight vector is $\omega = (2,3)$. We will only make the computations in the positive $x$ and $y$ charts, because in the negative charts the computations are analogous.

Our transformations are of the form
$$(x,y) = (\tilde{x}^{2},\tilde{x}^{3}\tilde{y}), \ (x,y) = (\bar{y}^{2}\bar{x},\bar{y}^{3}),$$
in the $x$ and $y$ directions, respectively. In the $x$-chart, the auxiliary vector field takes the form
$$\widetilde{X}_{A} = (\tilde{x}^{5}\tilde{y})\tilde{x}\displaystyle\frac{\tilde{y}}{2}\tilde{x}\displaystyle\frac{\partial}{\partial\tilde{x}} + (\tilde{x}^{5}\tilde{y})\tilde{x}\Big{(}\tilde{y}^{-1} - \displaystyle\frac{3}{2}\tilde{x}\tilde{y}\Big{)}\tilde{y}\displaystyle\frac{\partial}{\partial \tilde{y}},$$
and then multiplying such expression by $\tilde{x}^{-6}$ we have
$$\widetilde{X}_{A} = (\tilde{y})\displaystyle\frac{\tilde{y}}{2}\tilde{x}\displaystyle\frac{\partial}{\partial\tilde{x}} + (\tilde{y})\Big{(}1 - \displaystyle\frac{3}{2}\tilde{y}^{2}\Big{)}\displaystyle\frac{\partial}{\partial \tilde{y}},$$
which implies that we have the following constrained system in the positive $x$-direction:
$$\tilde{y}\dot{\tilde{x}} = \displaystyle\frac{\tilde{x}\tilde{y}}{2}, \ \tilde{y}\dot{\tilde{y}} = 1 - \displaystyle\frac{3\tilde{y}^{2}}{2}.$$

In this chart, the points $(0,\pm\sqrt{\frac{2}{3}})$ are hyperbolic saddles of the adjoint vector field and the impasse set intersects the exceptional divisor $\{x = 0\}$ at the origin. It is also interesting to observe what happens to the Newton polygon of the auxiliary vector field \eqref{exe-cusp-fold-auxiliary} when we blow-up the origin in this direction. The main segment $\gamma_{1}$ is sent to the vertical segment $\widetilde{\gamma}_{1}$ contained in the line $\{r = 6\}$ of the $rs$-plane. When we multiply the auxiliary vector field by $\tilde{x}^{-6}$, we translate the vertical segment $\widetilde{\gamma}_{1}$ to the $s$-axis. See Figure \ref{fig-exe-newton-cusp-dobra-x}.

Analogously, in the positive $y$-chart the auxiliary vector field \eqref{exe-cusp-fold-auxiliary} takes the form
$$\overline{X}_{A} = (\bar{x}\bar{y}^{5})\bar{y}\Big{(}\bar{x}^{-1} - \displaystyle\frac{2}{3}\bar{x}^{2}\Big{)}\bar{x}\displaystyle\frac{\partial}{\partial\bar{x}} + (\bar{x}\bar{y}^{5})\bar{y}\displaystyle\frac{\bar{x}^{2}\bar{y}}{3}\bar{y}\displaystyle\frac{\partial}{\partial \bar{y}},$$
and then multiplying such expression by $\bar{y}^{-6}$ we have
$$\overline{X}_{A} = (\bar{x})\Big{(}\bar{x}^{-1} - \displaystyle\frac{2}{3}\bar{x}^{2}\Big{)}\bar{x}\displaystyle\frac{\partial}{\partial\bar{x}} + (\bar{x})\displaystyle\frac{\bar{x}^{2}\bar{y}}{3}\bar{y}\displaystyle\frac{\partial}{\partial \bar{y}},$$
which implies that we have the following constrained system in the positive $y$-direction:
$$\bar{x}\dot{\bar{x}} = 1 - \displaystyle\frac{2}{3}\bar{x}^{3}, \ \bar{x}\dot{\bar{y}} = \displaystyle\frac{\bar{x}^{2}\bar{y}}{3}.$$

Observe that in this chart the impasse set intersects the exceptional divisor $\{y = 0\}$ at the origin and the point $\big{(}(\frac{3}{2})^{\frac{1}{3}},0\big{)}$ is a hyperbolic saddle of the adjoint vector field. The main segment $\gamma_{1}$ is sent to the horizontal segment $\bar{\gamma}_{1}$ contained in the line $\{s = 6\}$ of the $rs$-plane. When we multiply the auxiliary vector field by $\bar{y}^{-6}$, we translate the horizontal segment $\bar{\gamma}_{1}$ to the $r$-axis. See Figure \ref{fig-exe-newton-cusp-dobra-x}.

\begin{figure}[h]
  \center{\includegraphics[width=0.75\textwidth]{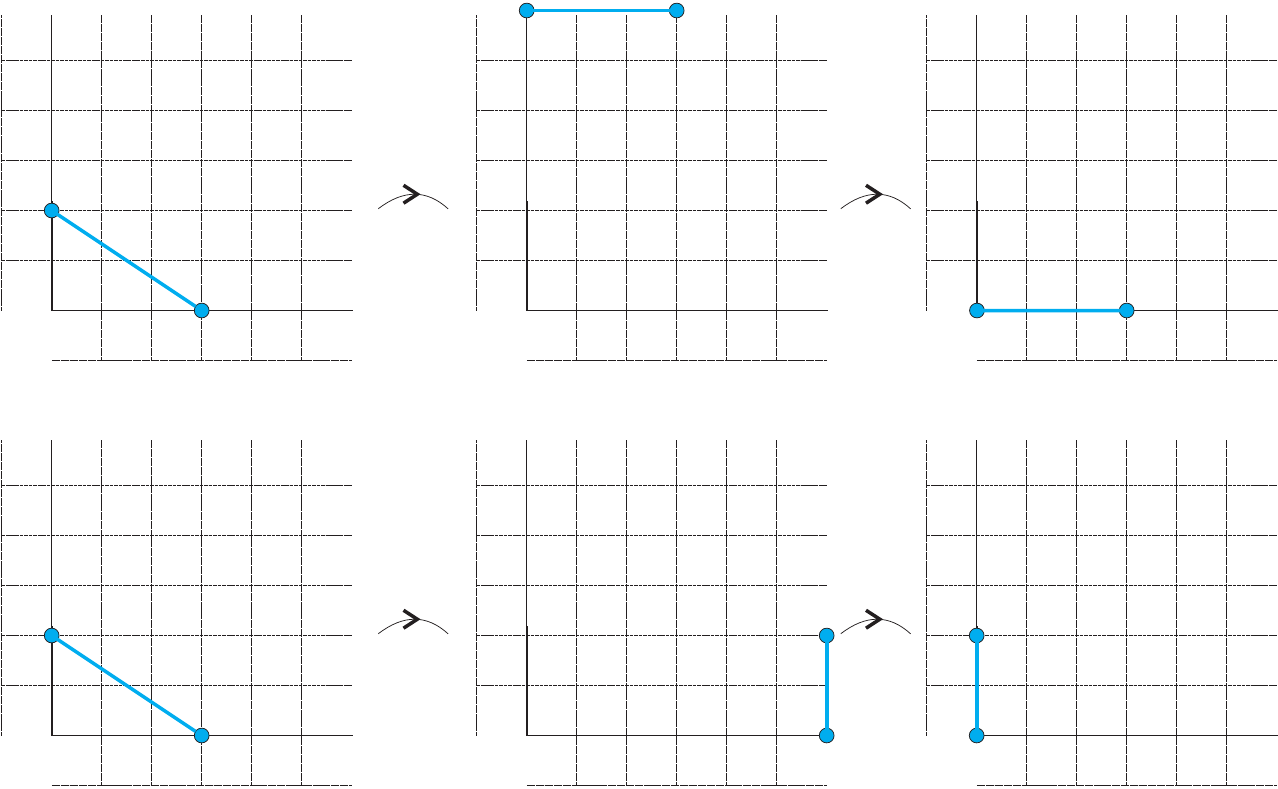}}\\
  \caption{Main segment after the blow-up in the $y$ direction (above) and after the blow-up in the $x$ direction (below).}\label{fig-exe-newton-cusp-dobra-x}
\end{figure}

Observe that all points in the exceptional divisor $\{x = 0\}$ are elementary, therefore the strict transformed of $\eqref{exe-cusp-fold}$ is elementary in $\{x = 0\}$. See Figure \ref{fig-exe-cusp-dobra}.

\begin{figure}[h]
  \center{\includegraphics[width=0.65\textwidth]{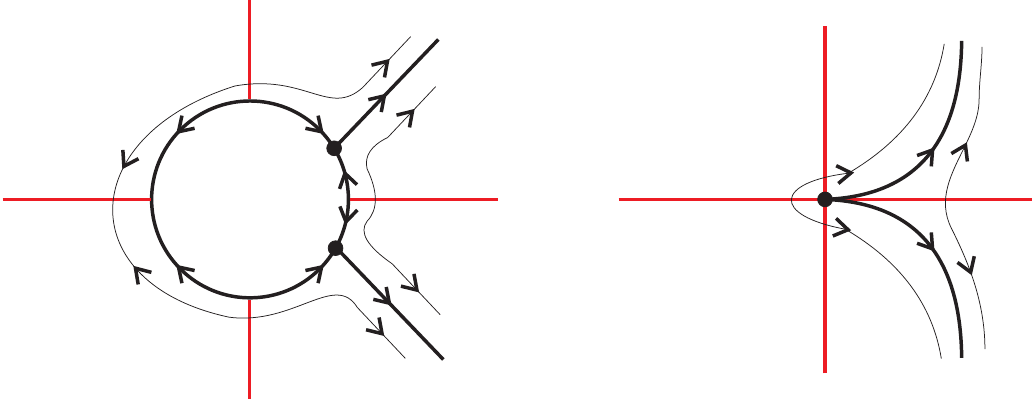}}\\
  \caption{Constrained system \eqref{exe-cusp-fold} after (left) and before (right) the weighted blow-up at the origin.}\label{fig-exe-cusp-dobra}
\end{figure}

\section{Properties of the Newton polygon: controllability and the notion of elementary point}\label{sec-properties}
\noindent

As we discussed previously, the Newton polygon strongly depends on the coordinates adopted. This means that after a blow up or a change of coordinates, we may obtain a completely different Newton polygon. The next goal is to defined tools in order to control the polygon during the resolution of singularities. More Precisely, we define tools to control the height of the Newton polygon. We also introduce a change of coordinates that will be useful during the resolution process (the so called admissible change of coordinates). In the end of this section, we show that non elementary points of a constrained system can be identified by means of the Newton polygon of its auxiliary vector field.

\subsection{Preparation of coordinates}
\noindent

We say that a coordinate system is \textbf{adapted to the divisor $\mathbb{E}$} if in such coordinate system the exceptional divisor is given by $\mathbb{E}_{x} = \{x = 0\}$, $\mathbb{E}_{y} = \{y = 0\}$ or $\mathbb{E}_{x,y} = \{xy = 0\}$ (see Definition 5.1, \cite{Dumortier}).

\begin{definition}
The Newton polygon $\mathcal{P}$ associated to a vector field is \textbf{controllable} if the main vertex $v_{0}$ is contained in $\{0\}\times\mathbb{N}$ or $\{-1\}\times\mathbb{N}$.
\end{definition}

Take the diagonalized constrained system
\begin{equation}\label{eq-impasse-2}
    \delta(x,y)\dot{x} = P(x,y), \ \delta(x,y)\dot{y} = Q(x,y),
\end{equation}
where $\delta$, $P$ and $Q$ are real analytic functions written as
\begin{center}
$\delta(x,y) = \sum c_{k,l}x^{k}y^{l}$, $P(x,y) = \sum a_{m,n}x^{m}y^{n}x$ and $Q(x,y) = \sum b_{m,n}x^{m}y^{n}y$.
\end{center}

The auxiliary vector field $X_{A}$ associated to \eqref{eq-impasse-2} is then given by
\begin{equation}\label{eq-aux-vec-field}
X_{A}(x,y) = \Bigg{(}\sum c_{k,l}x^{k}y^{l}\Bigg{)}\Bigg{(}\sum x^{m}y^{n}\big{(}a_{m,n}x\displaystyle\frac{\partial}{\partial x} + b_{m,n}y\displaystyle\frac{\partial}{\partial y}\big{)}   \Bigg{)}
\end{equation}

In \cite{Pelletier} was considered vector fields with isolated equilibrium points. This condition implies that $\partial\mathcal{P}$ intersects the $rs$-axes of the $rs$-plane. However, this is not always true when we consider vector fields with curves of singularities, which is the case of the auxiliary vector field. For example, consider the constrained system
$$
xy\dot{x} = x^{k}; \ xy\dot{y} = 2y^{l},
$$
whose auxiliary vector field is given by
$$
X_{A} = xyx^{k}\displaystyle\frac{\partial}{\partial x} +  xyy^{l}\displaystyle\frac{\partial}{\partial y}.
$$

The support of $X_{A}$ is then given by $\mathcal{Q} = \{(k,1), (1,l)\}$, and therefore $\partial\mathcal{P}$ does not intersect the $rs$-axes of the $rs$-plane.

In the resolution of singularities, we will assume that $\partial\mathcal{P}$ intersects at least the $s$-axis. The next lemma shows that this is not a restrictive condition.

\begin{lemma}\label{lemma-change-1}
Consider the constrained system \eqref{eq-impasse-2} near a point $p\not\in\partial\mathcal{M}$ and its auxiliary vector field \eqref{eq-aux-vec-field}. Define the change of coordinates
\begin{equation}\label{eq-lemma-change-1}
\Phi(\bar{x},\bar{y}) = (\bar{x} + \lambda\bar{y},\bar{y});
\end{equation}
with $\lambda\in\mathbb{R}$. Then the following statements are true:
\begin{enumerate}
  \item The diffeomorphism $\Phi$ gives us an orbital equivalence between \eqref{eq-impasse-2} and
\begin{equation}\label{lemma-change-1-eq-constrained}
\left\{
\begin{array}{rcl}
    \delta(\bar{x} + \lambda\bar{y},\bar{y})\dot{\bar{x}} & = & P(\bar{x} + \lambda\bar{y},\bar{y}) - \lambda Q(\bar{x} + \lambda\bar{y},\bar{y}), \\
    \delta(\bar{x} + \lambda\bar{y},\bar{y})\dot{\bar{y}} & = & Q(\bar{x} + \lambda\bar{y},\bar{y}).
  \end{array}
\right.
\end{equation}
  \item The diffeomorphism $\Phi$ gives us a topological equivalence between \eqref{eq-aux-vec-field} and the auxiliary vector field of \eqref{lemma-change-1-eq-constrained}.
  \item The Newton polygon $\mathcal{P}$ associated to the auxiliary vector field of \eqref{lemma-change-1-eq-constrained} has $(0,r)$ or $(-1,r)$ as main vertex, where $r$ is the lowest degree of its expansion.
\end{enumerate}
\end{lemma}
\begin{proof}
Observe that
\begin{center}
$\left\{
    \begin{array}{rcl}
      \bar{x} & = & x - \lambda y, \\
      \bar{y} & = & y,
    \end{array}
  \right.
$ $\Rightarrow$ $\left\{
    \begin{array}{rcl}
      \dot{\bar{x}} & = & \dot{x} - \lambda\dot{y}, \\
      \dot{\bar{y}} & = & \dot{y},
    \end{array}
  \right.$
\end{center}
and then we obtain the system \eqref{lemma-change-1-eq-constrained}. Since $\Phi$ is a diffeomorphism, it maps the impasse set of \eqref{lemma-change-1-eq-constrained} onto the impasse set of \eqref{eq-impasse-2}. Moreover, $\Phi$ maps phase curves of \eqref{lemma-change-1-eq-constrained} in phase curves of \eqref{eq-impasse-2}. This implies that \eqref{lemma-change-1-eq-constrained} and \eqref{eq-impasse-2} are orbitally equivalent. This also implies that \eqref{eq-aux-vec-field} and the auxiliary vector field of \eqref{lemma-change-1-eq-constrained} are topologically equivalent, provided outside the impasse set the phase curves of a constrained system and its auxiliary vector field are the same. Then items 1 and 2 are true.

After this coordinate change, the auxiliary vector field takes the form
$$
\scriptstyle
X_{A} = \Bigg{(}\sum c_{k,l}(\bar{x} + \lambda\bar{y})^{k}\bar{y}^{l}\Bigg{)}\Bigg{(}\sum (\bar{x} + \lambda\bar{y})^{m}\bar{y}^{n}\big{(}(a_{m,n}(\bar{x} + \lambda\bar{y}) - \lambda b_{m,n}\bar{y})\frac{\partial}{\partial \bar{x}} + b_{m,n}\bar{y}\frac{\partial}{\partial \bar{y}}\big{)}\Bigg{)}.
$$

Without loss of generality (and in order to simplify the notation), we will consider the first jet from now. Then we obtain
$$
\scriptstyle
\sum_{i = 0}^{k}c_{k,l}\binom{k}{i}\lambda^{i}\bar{x}^{k-i}\bar{y}^{l+ i}\Bigg{(}\sum_{j = 0}^{m}\binom{m}{j}\lambda^{j}\bar{x}^{m-j}\bar{y}^{n+j}\big{(}a_{m,n}\bar{x} + \lambda\bar{y}(a_{m,n} - b_{m,n})\big{)}\Bigg{)}\frac{\partial}{\partial \bar{x}}
$$
in the $\bar{x}$ component and
$$
\scriptstyle
\sum_{i = 0}^{k}c_{k,l}\binom{k}{i}\lambda^{i}\bar{x}^{k-i}\bar{y}^{l+ i}\Bigg{(}\sum_{j = 0}^{m}\binom{m}{j}\lambda^{j}\bar{x}^{m-j}\bar{y}^{n+j}b_{m,n}\Bigg{)}\bar{y}\frac{\partial}{\partial \bar{y}}
$$
in the $\bar{y}$ component.

If there is a pair $(m,n)$ such that $a_{m,n} - b_{m,n} \neq 0$, then the $\bar{x}$ component gives a contribution of the form $(-1,r)$ for the new support $\overline{\mathcal{Q}}$, where $r = (k+m) + (l+n)$. If $a_{m,n} - b_{m,n} = 0$ for all $(m,n)$, observe that at least one $b_{m,n}$ is nonzero and then the $\bar{y}$ component gives a contribution of the form $(0,r)$.
\end{proof}

By Lemma \ref{lemma-change-1} we can always suppose that the Newton polygon of the auxiliary vector field is controllable. Furthermore, the case where $p\in\partial \mathcal{M}$ gives us a controllable polygon in adapted coordinates since $\partial \mathcal{M}$ is invariant by the adjoint vector field.

\subsection{Admissible change of coordinates}
\noindent

In this section we consider another change of coordinates that will be useful during the resolution process. Such change of coordinates was called admissible in \cite{Pelletier}.

\begin{lemma}\label{lemma-change-2}
Consider a 2-dimensional constrained system near a point $p\in\mathcal{M}$, written in local coordinates as \eqref{eq-impasse-2}. Consider also its auxiliary vector field \eqref{eq-aux-vec-field}. Suppose also that its Newton polygon $\mathcal{P}_{1}$ is controllable. Define the change of coordinates
\begin{equation}\label{eq-lemma-change-2}
\Psi(\bar{x},\bar{y}) = (\bar{x}, - \alpha\bar{x}^{\beta} + \bar{y}),
\end{equation}
with $\alpha\in\mathbb{R}$ and $\beta\in\mathbb{N}$. Then the following statements are true:
\begin{enumerate}
  \item The diffeomorphism $\Psi$ is an orbital equivalence between \eqref{eq-impasse-2} and
\begin{equation}\label{lemma-change-2-eq-constrained}
\left\{
  \begin{array}{rcl}
    \delta(\bar{x},\bar{y} - \alpha\bar{x}^{\beta})\dot{\bar{x}} & = & P(\bar{x},\bar{y} - \alpha\bar{x}^{\beta}), \\
    \delta(\bar{x},\bar{y} - \alpha\bar{x}^{\beta})\dot{\bar{y}} & = & \alpha\beta P(\bar{x},\bar{y} - \alpha\bar{x}^{\beta})\bar{x}^{\beta-1} + Q(\bar{x},\bar{y} - \alpha\bar{x}^{\beta}).
  \end{array}
\right.
\end{equation}
  \item The diffeomorphism $\Psi$ is a topological equivalence between \eqref{eq-aux-vec-field} and the auxiliary vector field of \eqref{lemma-change-2-eq-constrained}.
  \item The Newton polygons $\mathcal{P}_{1}$ and $\mathcal{P}_{2}$ associated to the auxiliary vector field of \eqref{lemma-change-2-eq-constrained} and \eqref{eq-aux-vec-field}, respectively, have the same height.
  \item If the Newton polygon of \eqref{eq-aux-vec-field} has a non horizontal main segment, then the slope of the main segment increases after this change of coordinates.
  \item If before this change of coordinates the divisor is locally given by $\mathbb{E}_{x}$, then in this new coordinate system the divisor is still invariant by the adjoint vector field.
\end{enumerate}
\end{lemma}
\begin{proof}
Items 1 and 2 are proved applying the same ideas of the proof of the Lemma \ref{lemma-change-1}, so we will prove items 3, 4 and 5.

After a change of coordinates of the form \eqref{eq-lemma-change-2}, the auxiliary vector field associated to \eqref{lemma-change-2-eq-constrained} takes the form
$$
\scriptstyle
\Bigg{(}\sum c_{k,l}\bar{x}^{k}(\bar{y} - \alpha\bar{x}^{\beta})^{l}\Bigg{)}\Bigg{(}\sum\bar{x}^{m}(\bar{y} - \alpha\bar{x}^{\beta})^{n}a_{m,n}\Bigg{)}\bar{x}\frac{\partial}{\partial \bar{x}}
$$
in the $\bar{x}$ component and
$$
\scriptstyle
\Bigg{(}\sum c_{k,l}\bar{x}^{k}(\bar{y} - \alpha\bar{x}^{\beta})^{l}\Bigg{)}\Bigg{(}\sum\bar{x}^{m}(\bar{y} - \alpha\bar{x}^{\beta})^{n}\big{(}(\alpha\beta a_{m,n} -\alpha b_{m,n})\bar{x}^{\beta} + b_{m,n}\bar{y}\big{)}\Bigg{)}\frac{\partial}{\partial \bar{y}}
$$
in the $\bar{y}$ component. Once again, without loss of generality we will consider the first jet (in order to simplify the notation). Then we obtain
$$
\scriptstyle
\Bigg{(}\sum_{i = 0}^{l} c_{k,l}\binom{l}{i}(-\alpha)^{l-i}\bar{x}^{k+\beta(l-i)}\bar{y}^{i}\Bigg{)}\Bigg{(}\sum_{j=0}^{n}\binom{n}{j}(-\alpha)^{n-j}a_{m,n} \bar{x}^{m + \beta(n-j)}\bar{y}^{j}\Bigg{)}\bar{x}\frac{\partial}{\partial \bar{x}}
$$
in the $\bar{x}$ component and
$$
\scriptstyle
\Bigg{(}\sum_{i = 0}^{l} c_{k,l}\binom{l}{i}(-\alpha)^{l-i}\bar{x}^{k+\beta(l-i)}\bar{y}^{i}\Bigg{)}\Bigg{(}\sum_{j=0}^{n}\binom{n}{j}(-\alpha)^{n-j}\bar{x}^{m + \beta(n-j)}\bar{y}^{j}\big{(}(\alpha\beta a_{m,n} - \alpha b_{m,n})\bar{x}^{\beta} + b_{m,n}\bar{y}\big{)}\Bigg{)}\frac{\partial}{\partial \bar{y}}
$$
in the $\bar{y}$ component.

By assumption, the auxiliary vector field \eqref{eq-aux-vec-field} has $(-1,r)$ or $(0,r)$ as main vertex, where $r$ is the degree of the lowest graduation. This implies that there are $(k,l)$ and $(m,n)$ such that $(k+m) = -1$ (respectively $0$) and $(l+n) = r$. By these two expressions obtained previously, this property is still true for $\beta\in\mathbb{N}$, then we still have a point of the form $(-1,r)$ or $(0,r)$ as main vertex.

Since the main vertex was preserved, observe now that this operation increases the powers of the variable $x$ (because $\beta > 0$), then the $r$ entries of the points of the new support $\overline{\mathcal{Q}}$ are increased. Therefore, the points of the support (except the main vertex) are translated to the right and thus the slope of the main segment increases. See Figure \ref{fig-lemma-change-coord}.

\begin{figure}[h]
  \center{\includegraphics[width=0.50\textwidth]{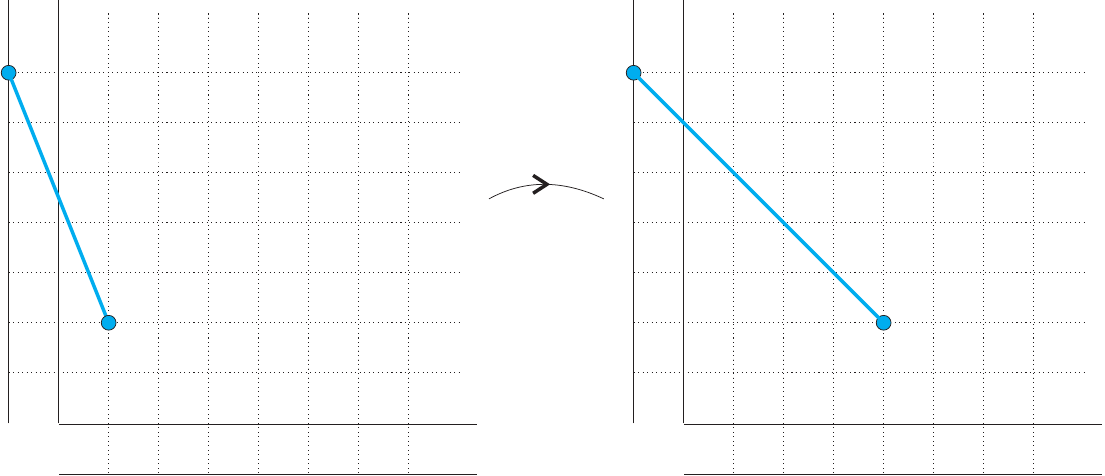}}\\
  \caption{Newton polygon of the auxiliary vector field after the change of coordinates \eqref{eq-lemma-change-2}.}\label{fig-lemma-change-coord}
\end{figure}

In the case where the divisor is locally given by $\mathbb{E}_{x}$, we have that the main vertex is of the form $(0,r)$ provided that the divisor is invariant by the adjoint vector field. After this coordinate change, the divisor is still locally given by $\mathbb{E}_{x}$. Moreover, by these two expressions of the auxiliar vector field obtained previously it follows that the divisor will be once again invariant by the adjoint vector field.
\end{proof}

\begin{definition}
The finite composition of change of coordinates of the form \eqref{eq-lemma-change-2} is called \textbf{admissible change of coordinates}.
\end{definition}

After an admissible change of coordinates, we obtain a coordinate system adapted to the divisor, the exceptional divisor is still invariant by the adjoint vector and the Newton polygon of the auxiliary vector field is still controllable.

\subsection{The relation between an elementary 2-dimensional constrained system and its Newton polygon}
\noindent

During the resolution of singularities, we must apply a finite number of operations (blow-ups) in order to obtain a ``simpler'' constrained system. This notion of ``simple'' comes from the notion of elementary constrained system (Definition \ref{def-desing2}). In what follows we show how to identify elementary points of the constrained system by means of the Newton polygon of the auxiliary vector field.

\begin{definition}\label{def-desing1}
A planar constrained differential system is \textbf{Newton elementary at $p$} if the Newton polygon $\mathcal{P}$ associated to the auxiliary vector field $X_{A}$ satisfies one of the following:
\begin{itemize}
  \item The main vertex of $\mathcal{P}$ is $(0,0)$, $(0,-1)$ or $(-1,0)$ (that is, if the height is less or equal to zero);
  \item The main segment $\gamma_{1}$ is horizontal.
\end{itemize}
\end{definition}

\textbf{Example:} We present an example to justify why we consider the case where $\gamma_{1}$ is horizontal as a final situation. Consider the constrained system
$$
y\dot{x} = 2x, \ y\dot{y} = -y.
$$

The support of the auxiliary vector field only contains the point $(0,1)$, and it is straightforward to see that the origin is an hyperbolic saddle of the adjoint vector field, where $\{x = 0\}$ and $\{y = 0\}$ are the separatrices. This means that $\{y = 0\}$ is at the same time a separatrix and the impasse set. It is not possible to change this configuration with a finite sequence of blowing-ups at the origin.

\begin{theorem}\label{teo-eq-definitions}
A planar constrained system is elementary at $p$ if, and only if, it is Newton elementary at $p$.
\end{theorem}
\begin{proof}
Fix a controllable coordinate system near $p$ and consider the auxiliary vector field
\begin{equation}\label{proof-equivalence}
\scriptstyle X_{A}(x,y) = \Bigg{(}\sum c_{k,l}x^{k}y^{l}\Bigg{)}\Bigg{(}\sum x^{m}y^{n}\big{(}a_{m,n}x\frac{\partial}{\partial x} + b_{m,n}y\frac{\partial}{\partial y}\big{)}\Bigg{)}.
\end{equation}

Firstly, suppose that the constrained system is Newton elementary at $p$. Observe that we do not have $m = -1$ and $n=-1$ at the same time in \eqref{proof-equivalence}. Take without loss of generality $m\geq 0$ and $n\geq -1$.

If $(0,0)\in\mathcal{P}$, by the expression \eqref{proof-equivalence} we have $(k + m) = 0$ and $(l+n) = 0$. It follows that $k = m = 0$ because $m \geq 0$. Since $(l+n) = 0$, $n\geq -1$ and $l\geq 0$, we have two cases to consider:
\begin{itemize}
  \item If $l = 1$ and $n = -1$, then $\delta$ is regular at $p$ and $p$ is not an equilibrium point of $X$. If $p\in\mathbb{E}_{x}$, then the impasse set $\Delta$ intersects the divisor transversally at $p$.
  \item If $l = n = 0$, then $p\not\in\Delta$. Moreover, since $a_{-1,1} = 0$ and $a_{00}^{2} + b^{2}_{00} \neq 0$, $p$ is an elementary equilibrium point of $X$.
\end{itemize}

It follows that if $(0,0)\in\mathcal{P}$, then $(\mathcal{M},\mathcal{X},\mathcal{I})$ is elementary at $p$.

Suppose now that $(0,-1)\in\mathcal{P}$. We have that $(k+m) = 0$ and $(l+n) = -1$, and this implies that $k = m = l = 0$ and $n = -1$. Therefore, $p\not\in\Delta$ and $p$ is not an equilibrium point of $X$. The proof for the case where $(-1,0)$ is the main vertex is completely analogous.

Finally, suppose that the main segment $\gamma_{1}$ is horizontal. Since the Newton polygon of the auxiliary vector field is obtained by means of the Minkowski sum between the support of the adjoint vector field and the impasse curve, then the Newton polygon of the impasse curve and the Newton polygon of the adjoint vector field have horizontal main segment. This implies that in this system of coordinates the impasse curve is given by
$$\delta = y\Bigg{(}c_{0,1}+\sum c_{k,l}x^{k}y^{l}\Bigg{)},$$
and we have the following cases, since $p$ is either an isolated equilibrium point or a regular point of the adjoint vector field $X$:
\begin{itemize}
  \item $a_{-1,0}\neq0$, which implies that $b_{n,-1} = 0$ for all $n$ and then the impasse curve coincides with a phase curve of $X$.
  \item $b_{0,-1}\neq0$, which implies that $X$ is transversal to $\Delta$ near the origin.
  \item $a_{00}^{2} + b_{00}^{2} \neq 0$, and then a separatrix of the equilibrium point (eventually the central manifold) coincides with the impasse curve.
\end{itemize}

In these three cases, the constrained system is elementary at $p$.

Now suppose that $(\mathcal{M},\mathcal{X},\mathcal{I})$ is elementary at $p$. Without loss of generality we can choose coordinates such that $\delta(x,y) = y$.

When $p$ is not an equilibrium point of $X$ and $p\not\in\Delta$, then $a^{2}_{-1,0} + b^{2}_{0,-1} \neq 0 \neq c_{00}$ and we have $(0,-1)\in\mathcal{P}$ or $(-1,0)\in\mathcal{P}$. If $p\in\Delta$ but $p$ is not an equilibrium point of $X$, then we have two possibilities:
\begin{itemize}
  \item The adjoint vector field is transversal to the impasse set $\Delta$ and therefore $c_{0,1}\neq 0$ and $b_{0,-1} \neq 0$, which implies that $(0,0)\in\partial\mathcal{P}$.
  \item The phase curve passing through $p$ coincides with the impasse set $\Delta$, and therefore $a_{-1,0}\neq0$ and $b_{n,-1}=0$ for all $n$. Then $a_{-1,0}\neq0\neq c_{0,1}$ gives the contribution $(-1,1)\in\mathcal{P}$ and the main segment of the polygon is horizontal.
\end{itemize}

If $p$ is a semi-hyperbolic equilibrium point of $X$ and $p\not\in\Delta$, we are in the classical case studied in \cite{Pelletier} and therefore $(0,0)\in\mathcal{P}$.

Finally, suppose that $p$ is a semi hyperbolic equilibrium point of $X$ and $p\in\Delta$. Remember that we are considering adapted coordinates and $\delta(x,y) = y$. By hypothesis, $\Delta$ must be contained in a separatrix (or the central manifold) of $p$. This means that such phase curve is not transversal to $\Delta$ and all coefficients of the form $b_{n,-1}$ are zero. On the other hand, $a_{0,0}^{2} + b_{0,0}^{2}\neq 0$, thus the main segment is horizontal and the statement is true.
\end{proof}

To end this section, we remark that when the divisor is locally given by $\mathbb{E}_{xy}$ we cannot have an elementary constrained system with $\Delta$ intersecting the origin. Indeed, since $\Delta$ is not contained in $\mathbb{E}_{x}\cup\mathbb{E}_{y}$ and $\Delta$ is transversal to $\mathbb{E}_{x}$ and $\mathbb{E}_{y}$ simultaneously, it follows that $c_{1,0}^{2} + c_{0,1}^{2} \neq 0$. Regarding that the origin in $\mathbb{E}_{xy}$ is always a hyperbolic equilibrium point of the adjoint vector field, we have that $a_{00}^{2} + b_{00}^{2} \neq 0$, hence $(1,0),(0,1)\in\mathcal{P}$ and then we must continue the resolution process.

\section{Resolution of singularities: local strategy and proof of the main Theorem}\label{sec-local-strategy}
\noindent

In this section we study the local theorem of the resolution of singularities. Firstly, we give the proof of the local theorem (which is basically the description of the algorithm of resolution of singularities) and later we give the proof of the lemmas and propositions used in the proof of Theorem \ref{teo-local-strategy}. We will focus only in the positives $x$ and $y$ charts because the computations on the negative charts are completely analogous. Remember that all singular points of the constrained differential system are isolated.

\begin{theorem}\label{teo-local-strategy}(Local Theorem)
Let $(\mathcal{M}, \mathcal{X}, \mathcal{I})$ be a 2-dimensional constrained system and let $p\in\mathcal{M}$. Fix local coordinates such that the Newton polygon of the auxiliary vector field $X_{A}$ is controllable. Then there is a finite sequence of weighted blow ups such that the strict transformed is elementary in all points of the excepcional divisor.
\end{theorem}

\begin{proof}
As usual, denote by $X$ the adjoint vector field in such local coordinates and by $\Delta$ the impasse set. Let $X_{A}$ be the auxiliary vector field. Firstly, we find the expression of the affine line $\{r\omega_{1} + s\omega_{2} = R\}$ that contains the main segment $\gamma_{1}$. If $\gamma_{1}$ is horizontal, we are done by Theorem \ref{teo-eq-definitions}. If this is not the case, denote the height of the Newton polygon by $h$ and define the weight vector as $\omega = (\omega_{1},\omega_{2})$. After a weighted blow-up in the the $y$ direction, it follows by Lemma \ref{lemma-blow-up-y-direction} that the origin of the excepcional divisor $\mathbb{E}_{y}$ is elementary. In other words, after just one weighted blow up to in the $y$ direction the origin of $\mathbb{E}_{y}$ is elementary.

Blowing up the point $p$ in the $x$ direction, by Lemma \ref{lemma-blow-up-x-direction} the new height $h_{1}$ of the Newton polygon in the origin of the divisor $\mathbb{E}_{x}$ is less than $h$, that is, the height decreased. Now, we need to study the other points of $\mathbb{E}_{x}$. Due to Proposition \ref{prop-after-blow-up}, there are two possibilities:
\begin{enumerate}
  \item There is a finite number $N$ of points $p_{i}$, $1\leq i \leq N$ in the exceptional divisor such that the height $h_{p_{i}}$ of the Newton polygon at $p_{i}$ is positive. Moreover, $h_{p_{i}} \leq h$ for all $i$.
  \item If there is a point in $\mathbb{E}_{x}$ such that $h_{p_{i}} = h$, this is the only point in the divisor that has positive height.
\end{enumerate}

In the first case, there is a finite number of points $p_{i}$ in the divisor such that or $p_{i}$ is not a semi-hyperbolic equilibrium point of the adjoint vector field or the impasse set intersects the divisor at $p_{i}$. Eventually, a point on the divisor can be an equilibrium point and an impasse point at the same time. We must continue blowing up the points in which the height is positive and the main segment is not horizontal. The weights are chosen in the same way as before.

For the second case, we remark that this point cannot be the origin. Indeed, we show that in this case the first graduation of the auxiliary vector field is very particular (see Proposition \ref{prop-before-blow-up-expressions}). So we apply an admissible change of coordinates, such as in Lemma \ref{lemma-change-2}. Such change of coordinates does not increase the height of the polygon, but it increases the slope of the main segment. Now we obtain a new weight vector to continue the process, due to Proposition \ref{prop-change-coord-infty}.

At each step, we have the same two possibilities. Since there is a finite number of non-elementary points on the divisor, the height is finite and at each step the height decreases (eventually applying admissible change of coordinates), after a finite number $n_{0}$ of blowing ups we have that all points on the divisor do not have positive height or the main segment is horizontal, that is, all points are elementary.

The worst scenario is when a finite composition of admissible change of coordinates is not sufficient to obtain a weight vector such that the blow-up such vector decreases the height. In this case, there is an analytic change of coordinates such that, in this new system of coordinates, the main segment is horizontal and then we finish the resolution process (see Proposition \ref{prop-change-coord-infty}).
\end{proof}

In the next subsection we show the lemmas and propositions used in the previous proof. Before such study, we discuss some similarities and differences between our approach and that one in \cite{Pelletier}.

Just as in the work of Pelletier \cite{Pelletier}, our proof is based on weighted blow ups and the Newton polygon. Therefore, many ideas presented here can be found in such work, for example, the behavior of the Newton polygon during the resolution process. Nevertheless, some definitions were inspired in \cite{Panazzolo}.

We recall that in \cite{Pelletier} the author considered analytic vector fields such that the origin in an isolated equilibrium point. Furthermore, in such paper the author proved that applying a finite sequence of blowing-ups in a singularity, we decrease the so called total multiplicity of a singularity (eventually applying some change of coordinates).

However, in this work we consider the Newton polygon of the auxiliary vector field, which is an analytic vector field that has a curve of equilibrium points. Here, we show that applying a finite sequence of blowing-ups in a singularity, we decrease the height of the Newton polygon (also applying some change of coordinates when necessary). In this sense, some ideas of our approach were inspired in \cite{BelottoPanazzolo}.

\subsection{The local strategy}
\noindent

From now on, fix local coordinates for the constrained differential system $(\mathcal{M},\mathcal{X},\mathcal{I})$, denote the local vector field by $X$ and the impasse set by $\Delta$. Let $X_{A}$ be the auxiliary vector field \eqref{eq-aux-vec-field} and assume that the Newton polygon is controllable and the main segment $\gamma_{1}$ is not horizontal. Recall that the weight vector $\omega = (\omega_{1}, \omega_{2})$ is chosen by the expression of the affine line $\{r\omega_{1} + s\omega_{2} = R\}$ that contains the main segment $\gamma_{1}$.

\begin{lemma}
After a blow-up with weight $\omega = (\omega_{1}, \omega_{2})$, the auxiliary vector field \eqref{eq-aux-vec-field} takes the form
\begin{equation}\label{lemma-eq-blowup-x}
\scriptstyle \widetilde{X}_{A}(\tilde{x},\tilde{y}) = \Bigg{(}\sum c_{k,l}\tilde{x}^{k\omega_{1} + l\omega_{2}}\tilde{y}^{l}\Bigg{)}\Bigg{(}\sum\tilde{x}^{m\omega_{1} + n\omega_{2}}\tilde{y}^{n}\big{(} \frac{a_{m,n}}{\omega_{1}}\tilde{x}\frac{\partial}{\partial \tilde{x}} + (b_{m,n} -\frac{\omega_{2}}{\omega_{1}}a_{m,n})\tilde{y}\frac{\partial}{\partial \tilde{y}}\big{)}\Bigg{)}
\end{equation}
\normalsize
in the $x$ direction and
\begin{equation}\label{lemma-eq-blowup-y}
\begin{aligned}
\scriptstyle \widetilde{X}_{A}(\tilde{x},\tilde{y}) = \Bigg{(}\sum c_{k,l}\tilde{x}^{k}\tilde{y}^{k\omega_{1} + l\omega_{2}}\Bigg{)}\Bigg{(}\sum \tilde{x}^{m}\tilde{y}^{m\omega_{1} + n\omega_{2}}\big{(}(a_{m,n} - \frac{\omega_{1}}{\omega_{2}}b_{m,n})\tilde{x}\frac{\partial}{\partial \tilde{x}} + \frac{b_{m,n}}{\omega_{2}}\tilde{y}\frac{\partial}{\partial \tilde{y}}\big{)}\Bigg{)}
\end{aligned}
\end{equation}
\normalsize
in the $y$ direction.
\end{lemma}
\begin{proof}
For the $x$ direction, observe that the blow-up
\begin{center}
$x = \tilde{x}^{\omega_{1}}$, $y = \tilde{x}^{\omega_{2}}\tilde{y}$
\end{center}
gives us the relation
$$
\left(
  \begin{array}{c}
    \dot{x} \\
    \dot{y} \\
  \end{array}
\right)
=
\left(
  \begin{array}{cc}
    \omega_{1}\tilde{x}^{\omega_{1} - 1} & 0 \\
    \omega_{2}\tilde{x}^{\omega_{2} - 1}\tilde{y} & \tilde{x}^{\omega_{2}} \\
  \end{array}
\right)
\left(
  \begin{array}{c}
    \dot{\tilde{x}} \\
    \dot{\tilde{y}} \\
  \end{array}
\right),
$$
and \eqref{eq-impasse-2} becomes the constrained system
$$
\left\{
  \begin{array}{rcl}
    \delta(\tilde{x}^{\omega_{1}}, \tilde{x}^{\omega_{2}}\tilde{y})\dot{\tilde{x}} & = & \displaystyle\sum\tilde{x}^{m\omega_{1} + n\omega_{2}}\tilde{y}^{n}\frac{a_{m,n}}{\omega_{1}} \tilde{x}, \\
    \delta(\tilde{x}^{\omega_{1}}, \tilde{x}^{\omega_{2}}\tilde{y})\dot{\tilde{y}} & = & \displaystyle\sum\tilde{x}^{m\omega_{1} + n\omega_{2}}\tilde{y}^{n}\big{(}b_{m,n} -\frac{\omega_{2}}{\omega_{1}}a_{m,n}\big{)}\tilde{y},
  \end{array}
\right.
$$
whose auxiliary vector field is of the form \eqref{lemma-eq-blowup-x}. Analogously we obtain the expression \eqref{lemma-eq-blowup-y} in the $y$ direction.
\end{proof}

\begin{lemma}\label{lemma-blow-up-y-direction}
After a blow-up in the $y$ direction, the Newton polygon is still controllable, the divisor $\mathbb{E}_{y}$ is invariant by the adjoint vector field $X$ and the origin is an elementary point of the constrained system.
\end{lemma}
\begin{proof}
Take the first level of the $(\omega_{1},\omega_{2})$-graduation
$$\scriptstyle X_{A,d}^{(\omega_{1},\omega_{2})}(x,y) = \Bigg{(}\sum_{\omega_{1} k + \omega_{2} l = d_{1}} c_{k,l}x^{k}y^{l}  \Bigg{)} \Bigg{(}\sum_{\omega_{1} m + \omega_{2} n = d_{2}} x^{m}y^{n}\Big{(}a_{m,n}x\frac{\partial}{\partial x} +  b_{m,n}y\frac{\partial}{\partial y}\Big{)}\Bigg{)},$$
where $d_{1} + d_{2} = d$. After a blow up in the $y$ direction the auxiliary vector field takes the from
$$
\scriptsize
\begin{aligned}
\widetilde{X}_{A}(\tilde{x},\tilde{y}) = {} &
\Bigg{(}\sum c_{k,l}\tilde{x}^{k}\tilde{y}^{k\omega_{1} + l\omega_{2}}\Bigg{)}\Bigg{(}\sum \tilde{x}^{m}\tilde{y}^{m\omega_{1} + n\omega_{2}}\big{(}(a_{m,n} - \frac{\omega_{1}}{\omega_{2}}b_{m,n})\tilde{x}\frac{\partial}{\partial \tilde{x}} + \frac{b_{m,n}}{\omega_{2}}\tilde{y}\frac{\partial}{\partial \tilde{y}}\big{)}\Bigg{)} = \\
& \Bigg{(}\sum c_{k,l}\tilde{x}^{k}\tilde{y}^{d_{1}}\Bigg{)}\Bigg{(}\sum \tilde{x}^{m}\tilde{y}^{d_{2}}\big{(}(a_{m,n} - \frac{\omega_{1}}{\omega_{2}}b_{m,n})\tilde{x}\frac{\partial}{\partial \tilde{x}} + \frac{b_{m,n}}{\omega_{2}}\tilde{y}\frac{\partial}{\partial \tilde{y}}\big{)}\Bigg{)}.
\end{aligned}
$$

In the $(r,s)$-plane, this operation sends lines of the form $\{r\omega_{1} + s\omega_{2} = R\}$ into horizonal lines. In other words, the main segment $\gamma_{1}$ is sent to a horizontal segment. Dividing the expression above by $\tilde{y}^{d}$, we move this horizontal segment to the $r$-axis (see Figure \ref{fig-lemma-blow-up-y}). Observe that the first entries of the points of the polygon do not change. In these new coordinates, the main vertex of the Newton polygon ate the origin is $(-1,0)$ or $(0,0)$. By Theorem \ref{teo-eq-definitions}, the origin is now an elementary point.
\end{proof}
\begin{figure}[h]
  \center{\includegraphics[width=0.75\textwidth]{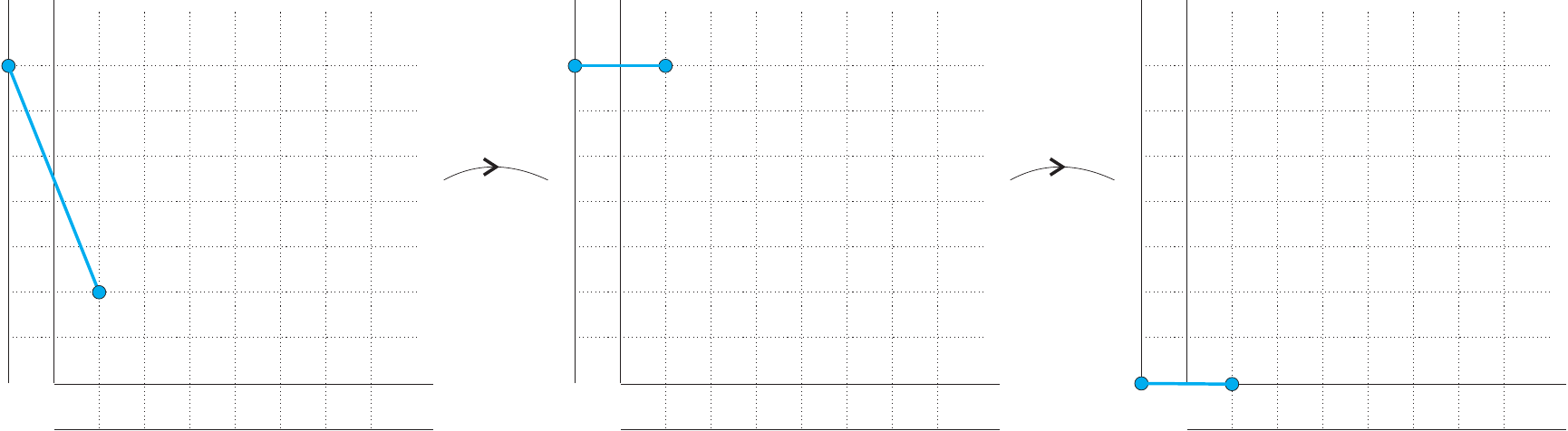}}\\
  \caption{Main segment of the Newton polygon after a blow up at the $y$ direction.}\label{fig-lemma-blow-up-y}
\end{figure}

\begin{lemma}\label{lemma-blow-up-x-direction}
After a blow-up in the $x$ direction, the main vertex is contained in $\{0\}\times\mathbb{N}$ (and therefore the Newton polygon is still controllable), the divisor $\mathbb{E}_{x}$ is invariant by the adjoint vector field $X$. Moreover, the height of the Newton polygon decreases.
\end{lemma}
\begin{proof}
We recall that the choice of the weight vector depends on the equation of the main segment $\gamma_{1}$. Once the choice is made, we take the first level of the $(\omega_{1},\omega_{2})$-graduation
$$\scriptstyle X_{A,d}^{(\omega_{1},\omega_{2})}(x,y) = \Bigg{(}\sum_{\omega_{1} k + \omega_{2} l = d_{1}} c_{k,l}x^{k}y^{l}  \Bigg{)} \Bigg{(}\sum_{\omega_{1} m + \omega_{2} n = d_{2}} x^{m}y^{n}\Big{(}a_{m,n}x\frac{\partial}{\partial x} +  b_{m,n}y\frac{\partial}{\partial y}\Big{)}\Bigg{)},$$
where $d_{1} + d_{2} = d$.

Firstly, we simplify the notation. It is sufficient to study the first level of the graduation. For each $l$, there is only one $k$ such that $\omega_{1} k + \omega_{2} l = d_{1}$. Analogously, for each $n$, there is only one $m$ such that $\omega_{1} m + \omega_{2} n = d_{2}$. So we rewrite the last expression as
$$\scriptstyle X_{A,d}^{(\omega_{1},\omega_{2})}(x,y) = \Bigg{(}\sum_{l = 0}^{r_{1}} c_{k,l}x^{k}y^{l}  \Bigg{)} \Bigg{(}\sum_{n = -1}^{r_{2}} x^{m}y^{n}\Big{(}a_{m,n}x\frac{\partial}{\partial x} +  b_{m,n}y\frac{\partial}{\partial y}\Big{)}\Bigg{)},$$
where $\omega_{2}r_{1} = d_{1}$ and $\omega_{2}r_{2} = d_{2}$ (if the main vertex is contained in $\{0\}\times\mathbb{N}$) or $\omega_{2}r_{2} = d_{2} + \omega_{1}$ (if the main vertex is contained in $\{-1\}\times\mathbb{N}$). From \eqref{lemma-eq-blowup-x}, after a blow-up in the $x$ direction, we have
$$
\scriptstyle
\Bigg{(}\sum_{l = 0}^{r_{1}} c_{k,l}\tilde{x}^{\omega_{1} k + \omega_{2} l}\tilde{y}^{l}  \Bigg{)} \Bigg{(}\sum_{n = -1}^{r_{2}}\tilde{x}^{\omega_{1} m + \omega_{2} n}\tilde{y}^{n}\Big{(}\frac{a_{m,n}}{\omega_{1}}\tilde{x}\frac{\partial}{\partial \tilde{x}} +  (b_{m,n} - \frac{\omega_{2}}{\omega_{1}}a_{m,n})\tilde{y}\frac{\partial}{\partial \tilde{y}}\Big{)}\Bigg{)}.
$$

Concerning the Newton polygon, this operation sends lines of the form $\{r\omega_{1} + s\omega_{2} = R\}$ into vertical lines. Moreover, the second entries of the points of the polygon do not change. Therefore, $\gamma_{1}$ is sent into a vertical segment. Now, multiplying the auxiliary vector field by $\tilde{x}^{-d}$, we get
\begin{equation}\label{eq-aux-vec-field-after-blowup}
\scriptstyle \Bigg{(}\sum_{l = 0}^{r_{1}} c_{k,l}\tilde{y}^{l}  \Bigg{)} \Bigg{(}\sum_{n = -1}^{r_{2}}\tilde{y}^{n}\Big{(}\frac{a_{m,n}}{\omega_{1}}\tilde{x}\frac{\partial}{\partial \tilde{x}} +  (b_{m,n} - \frac{\omega_{2}}{\omega_{1}}a_{m,n})\tilde{y}\frac{\partial}{\partial \tilde{y}}\Big{)}\Bigg{)}.
\end{equation}

This operation moves the vertical lines to the left. In particular, the main segment $\gamma_{1}$ is now contained in $\{ 0 \}\times \mathbb{N}$ (see Figure \ref{fig-lemma-blow-up-x}). Hence in this new coordinate system the main vertex has the form $(0,r)$ and the Newton polygon is controllable. Moreover, the divisor $\mathbb{E}_{x}$ is invariant by the adjoint vector field. Finally, since there is a vertex $v_{1} = (i_{1},j_{1})\in\gamma_{1}$ where $j_{1}$ is smaller than the ordinate the old main vertex $v_{0}$, the height of the Newton polygon decreased.
\end{proof}
\begin{figure}[h]
  \center{\includegraphics[width=0.75\textwidth]{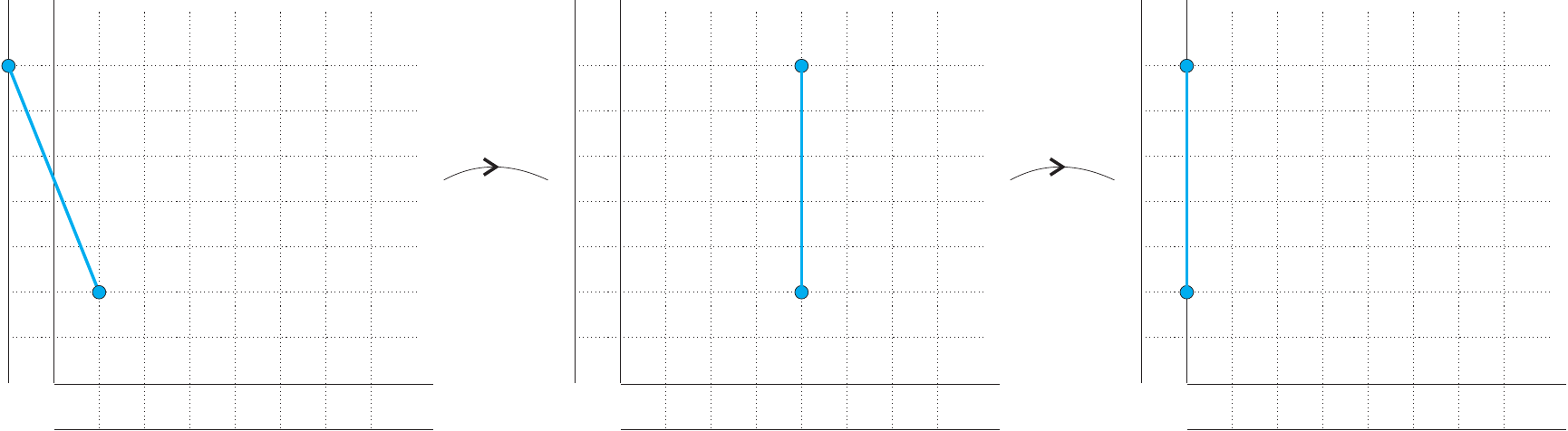}}\\
  \caption{Main segment of the Newton polygon after a blow up at the $x$ direction.}\label{fig-lemma-blow-up-x}
\end{figure}

Now, we must study what happens to the other points of the exceptional divisor.

\begin{lemma}\label{lemma-change-3}
Consider the auxiliary vector field \eqref{eq-aux-vec-field-after-blowup}, which is the auxiliary vector field after a blow-up in the $x$ direction. After a translation of the form
\begin{equation}\label{eq-lemma-change-3}
\Omega(\bar{x},\bar{y}) = (\bar{x},\bar{y} + \xi),
\end{equation}
where $\xi\in\mathbb{R}$, the auxiliary vector field takes the form (dropping the tildes and bars in order to simplify the notation)
\begin{equation}\label{eq-lemma-translation}
\scriptstyle X_{A}(x,y) = \Bigg{(}\sum_{i = 0}^{r_{1}}C_{i}(\xi)y^{i}\Bigg{)}\Bigg{(}\sum_{j = 0}^{r_{2}}\alpha_{j}(\xi)y^{j}x\frac{\partial}{\partial x} + \sum_{j = -1}^{r_{2}}\beta_{j}(\xi)y^{j}y\frac{\partial}{\partial y}\Bigg{)},
\end{equation}
where
\begin{center}
$\scriptstyle C_{i}(\xi) = \sum_{l = i}^{r_{1}}c_{l}\xi^{l-i}\binom{l}{i}$, $\scriptstyle \alpha_{j}(\xi) = \sum_{n = j}^{r_{2}}\frac{a_{n}}{\omega_{1}}\xi^{n-j}\binom{n}{j}$ and $\scriptstyle \beta_{j}(\xi) = \sum_{n = j}^{r_{2}}(b_{n} - \frac{\omega_{2}}{\omega_{1}}a_{n})\xi^{n-j}\binom{n+1}{j+1}$,
\end{center}
and such constants satisfy:
\begin{center}
$\scriptstyle \frac{\partial}{\partial \xi}C_{i}(\xi) = (i + 1)C_{i+1}(\xi)$, $\scriptstyle \frac{\partial}{\partial \xi}\alpha_{j}(\xi) = (j + 1)\alpha_{j+1}(\xi)$ and $\scriptstyle \frac{\partial}{\partial \xi}\beta_{j}(\xi) = (j + 2)\alpha_{j+1}(\xi)$.
\end{center}

Moreover, after this change of coordinates, the height of the Newton polygon does not increase, the main vertex is contained in $\{0\}\times\mathbb{N}$ and the exceptional divisor is still invariant by the adjoint vector field.
\end{lemma}
\begin{proof}
Take the equation \eqref{eq-aux-vec-field-after-blowup} and the translation \eqref{eq-lemma-change-3}. First, we recall the following property of sum:
$$\scriptstyle
\sum_{i = k}^{n}\sum_{j = k}^{i}z_{i,j} = \sum_{j = k}^{n}\sum_{i = j}^{n}z_{i,j},$$
where, $k\leq j\leq i\leq n$ are integers. In the $x$ component of the auxiliary vector field, we have (dropping the tildes):
$$
\scriptsize
\begin{aligned}\Bigg{(} {} & \sum_{l = 0}^{r_{1}}  c_{l}(y + \xi)^{l} \Bigg{)} \Bigg{(} \sum_{n = 0}^{r_{2}}(y + \xi)^{n}\Big{(}\frac{a_{n}}{\omega_{1}}\Big{)}\Bigg{)}x = \Bigg{(}\sum_{l = 0}^{r_{1}} c_{l}\sum_{i = 0}^{l}y^{i}\xi^{l-i}\binom{l}{i}  \Bigg{)} \Bigg{(}\sum_{n = 0}^{r_{2}}\frac{a_{n}}{\omega_{1}}\sum_{j = 0}^{n}y^{j}\xi^{n-j}\binom{n}{j}\Bigg{)}x = \\
& = \Bigg{(} \sum_{i = 0}^{r_{1}}y^{i}\sum_{l = i}^{r_{1}} c_{l}\xi^{l-i}\binom{l}{i}  \Bigg{)} \Bigg{(} \sum_{j=0}^{r_{2}}y^{j}\sum_{n = j}^{r_{2}}\frac{a_{n}}{\omega_{1}}\xi^{n-j}\binom{n}{j}\Bigg{)}x = \Bigg{(}\sum_{i = 0}^{r_{1}}y^{i}C_{i}(\xi)  \Bigg{)} \Bigg{(} \sum_{j= 0}^{r_{2}}y^{j}\alpha_{j}(\xi)\Bigg{)}x.
\end{aligned}
$$

In the $y$ component, we have (again dropping the tildes):
$$
\scriptsize
\begin{aligned}
\Bigg{(} \sum_{i = 0}^{r_{1}}y^{i}C_{i}(\xi)\Bigg{)} & \Bigg{(} \sum_{n = -1}^{r_{2}}\big{(}b_{n} - \frac{\omega_{2}}{\omega_{1}}a_{n}\big{)}(y+\xi)^{n+1}\Bigg{)} = \\
& = \Bigg{(} \displaystyle\sum_{i = 0}^{r_{1}}y^{i}C_{i}(\xi)\Bigg{)} \Bigg{(}   \displaystyle\sum_{n = -1}^{r_{2}}\big{(}b_{n} - \frac{\omega_{2}}{\omega_{1}}a_{n}\big{)}(y+\xi)^{n+1}\frac{1}{y}\Bigg{)}y = \\
& = \Bigg{(} \displaystyle\sum_{i = 0}^{r_{1}}y^{i}C_{i}(\xi)\Bigg{)}  \Bigg{(}   \displaystyle\sum_{n = -1}^{r_{2}}\big{(}b_{n} - \frac{\omega_{2}}{\omega_{1}}a_{n}\big{)}\sum_{j = 0}^{n+1}y^{j-1}\xi^{n+1-j}\binom{n+1}{j}\Bigg{)}y = \\
& = \Bigg{(} \displaystyle\sum_{i = 0}^{r_{1}}y^{i}C_{i}(\xi)\Bigg{)}  \Bigg{(}   \displaystyle\sum_{N = 0}^{r_{2}}\big{(}b_{N-1} - \frac{\omega_{2}}{\omega_{1}}a_{N-1}\big{)}\sum_{j = 0}^{N}y^{j-1}\xi^{N-j}\binom{N}{j}\Bigg{)}y = \\
& = \Bigg{(} \displaystyle\sum_{i = 0}^{r_{1}}y^{i}C_{i}(\xi)\Bigg{)}  \Bigg{(}   \displaystyle\sum_{j = 0}^{r_{2}}y^{j-1}\sum_{N = j}^{r_{2}}\big{(}b_{N-1} - \frac{\omega_{2}}{\omega_{1}}a_{N-1}\big{)}\xi^{N-j}\binom{N}{j}\Bigg{)}y = \\
& = \Bigg{(} \displaystyle\sum_{i = 0}^{r_{1}}y^{i}C_{i}(\xi)\Bigg{)}  \Bigg{(}   \displaystyle\sum_{j = -1}^{r_{2}}y^{j}\sum_{n = j}^{r_{2}}\big{(}b_{n} - \frac{\omega_{2}}{\omega_{1}}a_{n}\big{)}\xi^{n-j}\binom{n+1}{j+1}\Bigg{)}y = &\\
& = \Bigg{(} \displaystyle\sum_{i = 0}^{r_{1}}y^{i}C_{i}(\xi)\Bigg{)}  \Bigg{(}   \displaystyle\sum_{j = -1}^{r_{2}}y^{j}\beta_{j}(\xi)\Bigg{)}y.
\end{aligned}
$$

Note that
$$
\scriptsize
\begin{aligned}
 &\frac{\partial}{\partial \xi}C_{i}(\xi) = \frac{\partial}{\partial \xi}\Bigg{(}c_{i} + \sum_{l = i+1}^{r_{1}}c_{l}\xi^{l-i}\binom{l}{i}\Bigg{)}  =  \sum_{l = i+1}^{r_{1}}c_{l}(l-i)\xi^{l-(i+1)}\binom{l}{i} =  \\
& = \sum_{l = i+1}^{r_{1}}c_{l}\xi^{l-(i+1)}\frac{l!}{i!(l-(i+1))!} = \sum_{l = i+1}^{r_{1}}c_{l}\xi^{l-(i+1)}\frac{(i+1)l!}{(i+1)!(l-(i+1))!} = (i+1)C_{i+1}(\xi).
\end{aligned}
$$

The other expressions involving the derivative with respect to $\xi$ can be shown analogously. Observe that $C_{r_{1}} = c_{r_{1}} \neq 0$, and
\begin{center}
$\alpha_{r_{2}} = \displaystyle\frac{a_{r_{2}}}{\omega_{1}}$ and $\beta_{r_{2}} = \big{(}b_{r_{2}} - \displaystyle\frac{\omega_{2}}{\omega_{1}}a_{r_{2}}\big{)}$
\end{center}
does not vanish simultaneously. Therefore, the height of the Newton polygon does not increases. We remark that this operation does not change the first entries of the points of the polygon. It is also clear that the exceptional divisor is invariant by the adjoint vector field.
\end{proof}

\begin{proposition}\label{prop-after-blow-up}
After a blow-up at the $x$ direction, the following statements are true:
\begin{enumerate}
  \item The impasse set $\Delta$ intersects the exceptional divisor only in a finite number of points;
  \item There is a finite number of semi-hyperbolic equilibrium points of the adjoint vector field $X$ contained in the exceptional divisor;
  \item If there is a point such that the height has not decreased, this is the only point on the divisor that has positive height.
\end{enumerate}
\end{proposition}
\begin{proof}
From Lemma \ref{lemma-change-3}, thus the impasse set $\Delta$ is locally given by
$$\scriptstyle\sum_{i = 0}^{r_{1}}y^{i}C_{i}(\xi) =  C_{0}(\xi) + C_{1}(\xi)y + ... + C_{r_{1}}(\xi)y^{r_{1}}.$$

Observe that there is a finite number of points $(0,\xi_{k})$ on the exceptional divisor such that $C_{0}(\xi_{k}) = 0$. In other words, since there is a finite number of roots for $C_{0}(\xi)$, the impasse set $\Delta$ intersects the exceptional divisor in a finite number of points.

Concerning the adjoint vector field, let us notice that the polynomials
\begin{center}
$\scriptsize \alpha_{0}(\xi) = \sum_{n = 0}^{r_{2}}\frac{a_{n}}{\omega_{1}}\xi^{n}, \ \beta_{-1}(\xi) = \sum_{n = -1}^{r_{2}}(b_{n} - \frac{\omega_{2}}{\omega_{1}}a_{n})\xi^{n+1}$
\end{center}
are not identically zero simultaneously. Indeed, if this was the case, we would have $a_{n} = b_{n} = 0$ for all $-1 \leq n\leq r_{2}$ and it would contradict the fact that the translation does not increase the height of the Newton polygon. Hence $\beta_{-1}\equiv 0$ or $\alpha_{0}\equiv 0$.

If $\beta_{-1}\not\equiv 0$, this polynomial has a finite number of roots, thus by the expression of the auxiliary vector field \eqref{eq-lemma-translation} there is a finite number of equilibrium points $(0,\xi_{k})$ of the adjoint vector field in the exceptional divisor (eventually some of them are not elementary). If $\beta_{-1}\equiv 0$, then $\alpha_{0}\not\equiv 0$. In this case, all points of the exceptional divisor are equilibrium points of the adjoint vector field. The polynomial $\alpha_{0}$ has a finite number of roots, and then by \eqref{eq-lemma-translation} there is a finite number of points $(0,\xi_{k})$ that have positive height. In other words, there is a finite number of non elementary equilibrium points in the divisor.

Finally, we show that if there is a point in the excepcional divisor such that the height of the polygon has not decreased, this is the only point in the divisor with positive height. Let $(0,\xi_{0})$ be the point such that the height did not decrease. We recall that $\xi_{0} \neq 0$, since by Lemma \ref{lemma-blow-up-x-direction} the height at the origin decreases. Then we must have
$$
\scriptsize
\left\{
  \begin{array}{rcccccccccl}
    C_{0}(\xi_{0}) & = & ... & = & C_{r_{1} - 1}(\xi_{0}) & = & 0 & , & C_{r_{1}}(\xi_{0}) & \neq & 0;  \\
    \alpha_{0}(\xi_{0}) & = & ... & = & \alpha_{r_{2} - 1}(\xi_{0}) & = & 0 & , & \alpha_{r_{2}}(\xi_{0}) & \neq & 0; \\
    \beta_{-1}(\xi_{0}) & = & ... & = & \beta_{r_{2} - 1}(\xi_{0}) & = & 0 & , & \beta_{r_{2}}(\xi_{0}) & \neq & 0.
  \end{array}
\right.
$$

From Lemma \ref{lemma-change-3} we know that
\begin{center}
$\scriptstyle \frac{\partial}{\partial \xi}C_{i}(\xi) = (i + 1)C_{i+1}(\xi)$, $\scriptstyle \frac{\partial}{\partial \xi}\alpha_{j}(\xi) = (j + 1)\alpha_{j+1}(\xi)$ and $\scriptstyle \frac{\partial}{\partial \xi}\beta_{j}(\xi) = (j + 2)\alpha_{j+1}(\xi)$.
\end{center}

Then $\xi_{0}$ must be, simultaneously, root of multiplicity $r_{1}$ of $C_{0}$, root of multiplicity $r_{2}$ of $\alpha_{0}$ and root of multiplicity $r_{2} + 1$ of $\beta_{-1}$. By the degree of such polynomials, such root is unique, then $(0,\xi_{0})$ is the only point in the exceptional divisor with positive height.
\end{proof}

In the case where the height does not decrease, geometrically we have a point $p$ such that it is an equilibrium point of the adjoint vector field and an impasse point simultaneously.

\begin{proposition}\label{prop-before-blow-up-expressions}
Under the conditions of the Proposition \ref{prop-after-blow-up}, suppose that after a weighted blow up at the $x$ direction the height did not decrease. Then, before the blow up in the $x$ direction, the first graduation of the auxiliary vector field \eqref{eq-aux-vec-field} is of the form
\begin{equation}\label{prop-eq-before-blow-up-expressions-1}
\begin{aligned}
\scriptstyle X_{A} =
\scriptstyle \Bigg{(}c_{r_{1}}\big{(}y-\xi_{0}x^{\frac{\omega_{2}}{\omega_{1}}}\big{)}^{r_{1}}\Bigg{)}
\scriptstyle \Bigg{(}\Big{(}a_{r_{2}}\big{(}y-\xi_{0}x^{\frac{\omega_{2}}{\omega_{1}}}\big{)}^{r_{2}}\Big{)}x\frac{\partial}{\partial x} {} & \scriptstyle + \Big{(}y^{-1}\big{(}b_{r_{2}} - \frac{\omega_{2}}{\omega_{1}}a_{r_{2}}\big{)}(y-\xi_{0}x^{\frac{\omega_{2}}{\omega_{1}}})^{r_{2}+1} + \\
 & \scriptstyle + \frac{\omega_{2}}{\omega_{1}}a_{r_{2}}(y-\xi_{0}x^{\frac{\omega_{2}}{\omega_{1}}})^{r_{2}}\Big{)}y\frac{\partial}{\partial y}\Bigg{)}
\end{aligned}
\end{equation}
if the main vertex is contained in $\{0\}\times\mathbb{N}$ or
\begin{equation}\label{prop-eq-before-blow-up-expressions-2}
\scriptstyle
X_{A} = \Bigg{(}c_{r_{1}}\big{(}y-\xi_{0}x^{\frac{\omega_{2}}{\omega_{1}}}\big{)}^{r_{1}}\Bigg{)} \Bigg{(}\Big{(}a_{r_{2}}\big{(}y-\xi_{0}x^{\frac{\omega_{2}}{\omega_{1}}}\big{)}^{r_{2}}\Big{)}\frac{\partial}{\partial x} + \Big{(}-\frac{\omega_{2}}{\omega_{1}}a_{r_{2}}(-\xi_{0}x^{\frac{\omega_{2}}{\omega_{1}}-1})(y - \xi_{0}x^{\frac{\omega_{2}}{\omega_{1}}})^{r_{2}}{)}\frac{\partial}{\partial y}\Bigg{)}
\end{equation}
if the main vertex is contained in $\{-1\}\times\mathbb{N}$. In both cases, $\frac{\omega_{2}}{\omega_{1}}\in\mathbb{N}$.
\end{proposition}
\begin{proof}
Now we characterize the expression of the auxiliary vector field (and hence the constrained system) before the blow up.

\textbf{Case 1:} Before the blow-up, the main vertex of the Newton polygon is of the form $(0,d_{1}  + d_{2})$. We recall that $\xi_{0}\neq0$ because at the origin the height decreases. Remember also that $\omega_{1}r_{1} = d_{1}$ and $\omega_{2}r_{2} = d_{2}$. Moreover, $a_{-1,r_{2}} = 0$. Since at $(0,\xi_{0})$ the height does not decrease, it follows from what we have shown that
$$
\scriptsize
\left\{
  \begin{array}{rccccccl}
    \displaystyle\sum_{l = 0}^{r_{1}}c_{l}\xi^{l} & = & C_{0}(\xi) & = & c_{r_{1}}(\xi - \xi_{0})^{r_{1}} & = & c_{r_{1}}\displaystyle\sum_{j = 0}^{r_{1}}\xi^{j}(-\xi_{0})^{r_{1} - j}\displaystyle\binom{r_{1}}{j}; \\
    \displaystyle\sum_{n = 0}^{r_{2}}\frac{a_{n}}{\omega_{1}}\xi^{n} & = & \alpha_{0}(\xi) & = & \frac{a_{r_{2}}}{\omega_{1}}(\xi - \xi_{0})^{r_{2}} & = & \frac{a_{r_{2}}}{\omega_{1}}\displaystyle\sum_{j = 0}^{r_{2}}\xi^{j}(-\xi_{0})^{r_{2} - j}\displaystyle\binom{r_{2}}{j}; \\
    \displaystyle\sum_{n = -1}^{r_{2}}\lambda_{n}\xi^{n+1} & = & \beta_{-1}(\xi) & = & \lambda_{r_{2}}(\xi - \xi_{0})^{r_{2}+1} & = & \lambda_{r_{2}}\displaystyle\sum_{j = 0}^{r_{2}+1}\xi^{j}(-\xi_{0})^{r_{2}+1 - j}\displaystyle\binom{r_{2}+1}{j}; \\
  \end{array}
\right.
$$
where $\lambda_{n} = \big{(}b_{n} - \frac{\omega_{2}}{\omega_{1}}a_{n}\big{)}$. Thus we have the following equalities, for each $l$ and $n$:
$$
\scriptsize
\left\{
    \begin{array}{rcl}
      c_{l} & = & c_{r_{1}}(-\xi_{0})^{r_{1} - l}\displaystyle\binom{r_{1}}{l}; \\
      a_{n} & = & a_{r_{2}}(-\xi_{0})^{r_{2} - n}\displaystyle\binom{r_{2}}{n};\\
      b_{n} & = & \big{(}b_{r_{2}} - \frac{\omega_{2}}{\omega_{1}}a_{r_{2}}\big{)}(-\xi_{0})^{r_{2} - n}\displaystyle\binom{r_{2} + 1}{n + 1} + \frac{\omega_{2}}{\omega_{1}}a_{r_{2}}(-\xi_{0})^{r_{2} - n}\displaystyle\binom{r_{2}}{n}.
    \end{array}
\right.
$$

We remark that for each $k$ and each $m$ we have $k = \frac{\omega_{2}}{\omega_{1}}(r_{1} - l)$ and $m = \frac{\omega_{2}}{\omega_{1}}(r_{2} - n)$. Then, before the blow up at the $x$ direction, the first graduation the auxiliary vector field is of the form
$$
\begin{aligned}
\scriptstyle \Bigg{(} & \sum_{l = 0}^{r_{1}} \scriptstyle c_{l}x^{k}y^{l}\Bigg{)}\Bigg{(}\sum_{n = 0}^{r_{2}}a_{n}x^{m}y^{n}\Bigg{)}x\frac{\partial}{\partial x} = \\
& \scriptstyle = \Bigg{(}\sum_{l = 0}^{r_{1}}c_{r_{1}}(-\xi_{0})^{r_{1} - l}\binom{r_{1}}{l}(x^{\frac{\omega_{2}}{\omega_{1}}})^{r_{1} - l}y^{l}\Bigg{)}\Bigg{(}\sum_{n = 0}^{r_{2}}a_{r_{2}}(-\xi_{0})^{r_{2} - n}\binom{r_{2}}{n}(x^{\frac{\omega_{2}}{\omega_{1}}})^{r_{2} - n}y^{n}\Bigg{)}x\frac{\partial}{\partial x} = \\
& \scriptstyle = \Bigg{(}\sum_{l = 0}^{r_{1}}c_{r_{1}}\binom{r_{1}}{l}(-\xi_{0}x^{\frac{\omega_{2}}{\omega_{1}}})^{r_{1} - l}y^{l}\Bigg{)}\Bigg{(}\sum_{n = 0}^{r_{2}}a_{r_{2}}\binom{r_{2}}{n}(-\xi_{0}x^{\frac{\omega_{2}}{\omega_{1}}})^{r_{2} - n}y^{n}\Bigg{)}x\frac{\partial}{\partial x} = \\
& \scriptstyle = \Bigg{(}c_{r_{1}}\big{(}y-\xi_{0}x^{\frac{\omega_{2}}{\omega_{1}}}\big{)}^{r_{1}}\Bigg{)}\Bigg{(}a_{r_{2}}\big{(}y-\xi_{0}x^{\frac{\omega_{2}}{\omega_{1}}}\big{)}^{r_{2}}\Bigg{)}x\frac{\partial}{\partial x}
\end{aligned}
$$
in the $x$ component, and in the $y$ component we have (dropping the expression of the impasse set in order to simplify the notation)

$$
\begin{aligned}
\scriptstyle
&\Bigg{(} \scriptstyle \sum_{n = -1}^{r_{2}}\Big{(}\big{(}b_{r_{2}} - \frac{\omega_{2}}{\omega_{1}}a_{r_{2}}\big{)}(-\xi_{0})^{r_{2} - n}\binom{r_{2} + 1}{n + 1} + \frac{\omega_{2}}{\omega_{1}}a_{r_{2}}(-\xi_{0})^{r_{2} - n}\binom{r_{2}}{n}\Big{)}x^{\frac{\omega_{2}}{\omega_{1}}(r_{2} - n)}y^{n}\Bigg{)}y\frac{\partial}{\partial y} = \\
& \scriptstyle = \Bigg{(}\sum_{n = -1}^{r_{2}}\big{(}b_{r_{2}} - \frac{\omega_{2}}{\omega_{1}}a_{r_{2}}\big{)}y^{n}(-\xi_{0}x^{\frac{\omega_{2}}{\omega_{1}}})^{r_{2} - n}\binom{r_{2} + 1}{n + 1} + \sum_{n = 0}^{r_{2}}\frac{\omega_{2}}{\omega_{1}}a_{r_{2}}y^{n}(-\xi_{0}x^{\frac{\omega_{2}}{\omega_{1}}})^{r_{2} - n}\binom{r_{2}}{n}\Bigg{)}y\frac{\partial}{\partial y} = \\
& \scriptstyle = \Bigg{(}\sum_{N = 0}^{r_{2}}\big{(}b_{r_{2}} - \frac{\omega_{2}}{\omega_{1}}a_{r_{2}}\big{)}y^{N-1}(-\xi_{0}x^{\frac{\omega_{2}}{\omega_{1}}})^{r_{2}+1 - N}\binom{r_{2} + 1}{N} + \frac{\omega_{2}}{\omega_{1}}a_{r_{2}}(y-\xi_{0}x^{\frac{\omega_{2}}{\omega_{1}}})^{r_{2}}\Bigg{)}y\frac{\partial}{\partial y} = \\
& = \scriptstyle \Bigg{(}y^{-1}\big{(}b_{r_{2}} - \frac{\omega_{2}}{\omega_{1}}a_{r_{2}}\big{)}(y-\xi_{0}x^{\frac{\omega_{2}}{\omega_{1}}})^{r_{2}+1} + \frac{\omega_{2}}{\omega_{1}}a_{r_{2}}(y-\xi_{0}x^{\frac{\omega_{2}}{\omega_{1}}})^{r_{2}}\Bigg{)}y\frac{\partial}{\partial y}.
\end{aligned}
$$

To end the Case 1, let us remark that $c_{r_{1}}\neq 0$. Then
\begin{itemize}
  \item If $a_{r_{2}}\neq0$, then the $x$ component gives the contributions $(0,r_{1} + r_{2})$ and $(\frac{\omega_{2}}{\omega_{1}}, r_{1} + r_{2} - 1)$ to the support $\mathcal{Q}$.
  \item If $a_{r_{2}} = 0$, then $b_{r_{2}}\neq0$ and the $y$ component gives contributions $(0,r_{1} + r_{2})$ and $(\frac{\omega_{2}}{\omega_{1}}, r_{1} + r_{2} - 1)$ to the support $\mathcal{Q}$.
\end{itemize}

In both cases, we have $\frac{\omega_{2}}{\omega_{1}}\in\mathbb{N}$.

\textbf{Case 2:} Suppose that the main vertex is of the form $(-1, d_{1} + d_{2})$. Remember that $\omega_{1}r_{1} = d_{1}$ and $\omega_{2}r_{2} = d_{2} + \omega_{1}$. We also remark that $a_{r_{2}} = a_{-1,r_{2}}\neq 0$ and $b_{r_{2}} = b_{-1,r_{2}} = 0$. This case is almost analogous to Case 1, with one more term in the expansion of the polynomials $\alpha_{0}(\xi)$ and $\beta_{-1}(\xi)$. The expression of $C_{0}$ remains exactly the same, so we focus in the other expressions. For each $k$ and each $m$ we have $k = \frac{\omega_{2}}{\omega_{1}}(r_{1} - l)$ and $m = \frac{\omega_{2}}{\omega_{1}}(r_{2} - n) - 1$. So in the $x$ component we obtain
$$
\begin{aligned}
\scriptstyle
\Bigg{(}\sum_{l = 0}^{r_{1}}c_{l}x^{k}y^{l}\Bigg{)}& \scriptstyle \Bigg{(}\sum_{n = 0}^{r_{2}}a_{n}x^{m}y^{n}\Bigg{)}x\frac{\partial}{\partial x} = \\
& \scriptstyle = \Bigg{(}c_{r_{1}}\big{(}y-\xi_{0}x^{\frac{\omega_{2}}{\omega_{1}}}\big{)}^{r_{1}}\Bigg{)}\Bigg{(}\sum_{n = 0}^{r_{2}}a_{r_{2}}(-\xi_{0})^{r_{2} - n}\binom{r_{2}}{n}(x^{\frac{\omega_{2}}{\omega_{1}}})^{r_{2} - n}x^{-1}y^{n}\Bigg{)}x\frac{\partial}{\partial x} = \\
& \scriptstyle = \Bigg{(}c_{r_{1}}\big{(}y-\xi_{0}x^{\frac{\omega_{2}}{\omega_{1}}}\big{)}^{r_{1}}\Bigg{)}\Bigg{(}\sum_{n = 0}^{r_{2}}a_{r_{2}}\binom{r_{2}}{n}(-\xi_{0}x^{\frac{\omega_{2}}{\omega_{1}}})^{r_{2} - n}y^{n}\Bigg{)}\frac{\partial}{\partial x} = \\
& \scriptstyle = \Bigg{(}c_{r_{1}}\big{(}y-\xi_{0}x^{\frac{\omega_{2}}{\omega_{1}}}\big{)}^{r_{1}}\Bigg{)}\Bigg{(}a_{r_{2}}\big{(}y-\xi_{0}x^{\frac{\omega_{2}}{\omega_{1}}}\big{)}^{r_{2}}\Bigg{)}\frac{\partial}{\partial x}.
\end{aligned}
$$

Before we compute the expression in the $y$ component, we must give the expression of the coefficients $b_{n}$. Since $b_{r_{2}} = 0$, then for each $n$ we have
$$
\begin{aligned}
\scriptstyle b_{n} & \scriptstyle = - \frac{\omega_{2}}{\omega_{1}}a_{r_{2}}(-\xi_{0})^{r_{2} - n}\binom{r_{2} + 1}{n + 1} + \frac{\omega_{2}}{\omega_{1}}a_{r_{2}}(-\xi_{0})^{r_{2} - n}\binom{r_{2}}{n} \scriptstyle = - \frac{\omega_{2}}{\omega_{1}}a_{r_{2}}(-\xi_{0})^{r_{2} - n}\Bigg{(}\binom{r_{2} + 1}{n + 1} - \binom{r_{2}}{n}\Bigg{)} = \\
& \scriptstyle = - \frac{\omega_{2}}{\omega_{1}}a_{r_{2}}(-\xi_{0})^{r_{2} - n}\binom{r_{2}}{n}\frac{r_{2}-n}{n+1} \scriptstyle = - \frac{\omega_{2}}{\omega_{1}}a_{r_{2}}(-\xi_{0})^{r_{2} - n}\binom{r_{2}}{n+1},
\end{aligned}
$$
and finally we obtain
$$
\begin{aligned}
\scriptstyle \Bigg{(}& \sum_{l = 0}^{r_{1}} \scriptstyle c_{l}x^{k}y^{l}\Bigg{)} \scriptstyle \Bigg{(}\sum_{n = -1}^{r_{2}}b_{n}x^{m}y^{n}\Bigg{)}y\frac{\partial}{\partial y} = \\
& \scriptstyle = \Bigg{(}c_{r_{1}}\big{(}y-\xi_{0}x^{\frac{\omega_{2}}{\omega_{1}}}\big{)}^{r_{1}}\Bigg{)} \Bigg{(}\sum_{n = -1}^{r_{2}} \Big{(} - \frac{\omega_{2}}{\omega_{1}}a_{r_{2}}(-\xi_{0})^{r_{2} - n}\binom{r_{2}}{n+1}\Big{)} x^{\frac{\omega_{2}}{\omega_{1}}(r_{2} - n)}x^{-1}y^{n}\Bigg{)}y\frac{\partial}{\partial y} = \\
& \scriptstyle = \Bigg{(}c_{r_{1}}\big{(}y-\xi_{0}x^{\frac{\omega_{2}}{\omega_{1}}}\big{)}^{r_{1}}\Bigg{)} \Bigg{(}-\frac{\omega_{2}}{\omega_{1}}a_{r_{2}}x^{-1}\sum_{n = -1}^{r_{2}}(-\xi_{0})^{r_{2} - n}\binom{r_{2}}{n+1}(x^{\frac{\omega_{2}}{\omega_{1}}})^{(r_{2} - n)}y^{n}\Bigg{)}y\frac{\partial}{\partial y} = \\
\end{aligned}
$$
$$
\begin{aligned}
& \scriptstyle = \Bigg{(}c_{r_{1}}\big{(}y-\xi_{0}x^{\frac{\omega_{2}}{\omega_{1}}}\big{)}^{r_{1}}\Bigg{)} \Bigg{(}-\frac{\omega_{2}}{\omega_{1}}a_{r_{2}}x^{-1}\sum_{N = 0}^{r_{2}}\binom{r_{2}}{N}(-\xi_{0}x^{\frac{\omega_{2}}{\omega_{1}}})^{(r_{2} - N + 1)}y^{N-1}\Bigg{)}y\frac{\partial}{\partial y} = \\
& \scriptstyle = \Bigg{(}c_{r_{1}}\big{(}y-\xi_{0}x^{\frac{\omega_{2}}{\omega_{1}}}\big{)}^{r_{1}}\Bigg{)} \Bigg{(}-\frac{\omega_{2}}{\omega_{1}}a_{r_{2}}(-\xi_{0}x^{\frac{\omega_{2}}{\omega_{1}}-1})\sum_{N = 0}^{r_{2}}\binom{r_{2}}{N}(-\xi_{0}x^{\frac{\omega_{2}}{\omega_{1}}})^{(r_{2} - N)}y^{N}\Bigg{)}\frac{\partial}{\partial y} = \\
& \scriptstyle = \Bigg{(}c_{r_{1}}\big{(}y-\xi_{0}x^{\frac{\omega_{2}}{\omega_{1}}}\big{)}^{r_{1}}\Bigg{)} \Bigg{(}-\frac{\omega_{2}}{\omega_{1}}a_{r_{2}}(-\xi_{0}x^{\frac{\omega_{2}}{\omega_{1}}-1})(y - \xi_{0}x^{\frac{\omega_{2}}{\omega_{1}}})^{r_{2}}\Bigg{)}\frac{\partial}{\partial y}.
\end{aligned}
$$

From the expression of the $x$ component, the Newton polygon contains the point $(-1, r_{1} + r_{2})$. Moreover, the $y$ component gives the contribution $(\frac{\omega_{2}}{\omega_{1}} -1, r_{1} + r_{2} - 1)$. This implies $\frac{\omega_{2}}{\omega_{1}}\in\mathbb{N}$.
\end{proof}

\begin{proposition}\label{prop-change-coord-infty}
Suppose that the main segment $\gamma_{1}$ is not horizontal and after a blow-up at the $x$ direction there is a point in the excepcional divisor such that the height did not decrease. Then there is an analytic change of coordinates such that one of the following holds:
\begin{itemize}
  \item We obtain a new weight vector $\omega = (\omega_{1},\omega_{2})$ such that, after a blow up in the $x$ direction, the main vertex is contained in $\{0\}\times\mathbb{N}$, the height decreases and there is only a finite number of points in the exceptional divisor $\mathbb{E}_{x}$ such that they may not be elementary.
  \item The main segment is horizontal.
\end{itemize}
\end{proposition}
\begin{proof}
We know that there is only one non-elementary point $p$ in the exceptional divisor and the expression of the auxiliary vector field before the blow up is given by Proposition \ref{prop-before-blow-up-expressions}. We have two cases to consider:

\textbf{Case 1:} The point $p$ is an equilibrium point of the adjoint vector field. In this case, take a change of coordinates of the form
$$\scriptsize\Psi_{1}(x,y) = (x, y-\xi_{0}x^{\frac{\omega_{2}}{\omega_{1}}}).$$

We recall that by Proposition \ref{prop-before-blow-up-expressions} we have $\frac{\omega_{2}}{\omega_{1}}\in\mathbb{N}$. By Lemma \ref{lemma-change-2}, such change of coordinates does not change the height of the polygon and increases the slope of the main segment. Applying such change of coordinates in the expressions obtained from Proposition \ref{prop-before-blow-up-expressions}, the first graduation of the auxiliary vector field becomes
$$\scriptstyle X_{A} = c_{r_{1}}y^{r_{1}}\Bigg{(}a_{r_{2}}y^{r_{2}}x\frac{\partial}{\partial x} + b_{r_{2}}y^{r_{2}}y\frac{\partial}{\partial y}\Bigg{)} + ...
$$
if the main vertex is contained in $\{0\}\times\mathbb{N}$ or
$$
\scriptstyle X_{A} = c_{r_{1}}y^{r_{1}}\Bigg{(}a_{r_{2}}y^{r_{2}}\frac{\partial}{\partial x} + \frac{\omega_{2}}{\omega_{1}}a_{r_{2}}(\xi_{0}x^{\frac{\omega_{2}}{\omega_{1}}-1})y^{r_{2}}\frac{\partial}{\partial y}\Bigg{)} + ...
$$
if the main vertex is contained in $\{-1\}\times\mathbb{N}$. In both cases, the main segment of the new Newton polygon is not horizontal, and then we obtain new weights to continue the resolution of singularities, since the equilibrium point of the adjoint vector field is isolated.

Now we must blow up the origin with these new weights. After this new blow-up, once again we have two possibilities: the new blow up decreases the height of the Newton polygon or it does not decrease the height. In this last case, we apply again a change of coordinates $\Psi_{2}$ given by Lemma \ref{lemma-change-2}. Observe that this operation does not change the height and it increases the slope of the main segment. And then we continue the reasoning.

If it is necessary, we apply $n$ change of coordinates $\Psi_{n}$ given by Lemma \ref{lemma-change-2}. The main segment will not be horizontal and we obtain new weights to continue the resolution of singularities. These new weights decreases the height of the Newton polygon and then the Proposition is true.

In this case where $p$ is an equilibrium point of the adjoint vector field, we affirm that this iteration of admissible change of coordinates ends. Indeed, if this was not the case, since at each step the slope of the main segment increases, after an infinite composition of admissible coordinates
$$\scriptsize \Psi_{\infty}(x,y) = \Big{(}x, y-\sum_{i} \xi_{i}x^{\tilde{\omega}_{i}}\Big{)}$$
we should have a Newton polygon with horizontal main segment. Moreover, after the change of coordinates $\Psi_{\infty}$ we would have that the adjoint vector field is $\infty$-flat along the impasse set $\Delta$, which is a contradiction (see Lemmas 7 and 9 of \cite{Pelletier}). Then the iteration of admissible change of coordinates does not continue indefinitely.

\textbf{Case 2:} The point $p$ is not an equilibrium point of the adjoint vector field. This implies that $r_{2} = 0$ and the main vertex is not contained in $\{0\}\times\mathbb{N}$. Applying an admissible change of coordinates $\Psi_{1}$ in the expressions obtained from Proposition \ref{prop-before-blow-up-expressions}, the first graduation of the auxiliary vector field becomes
$$
\scriptstyle X_{A} = c_{r_{1}}y^{r_{1}}\Bigg{(}a_{-1,0}\frac{\partial}{\partial x} + \frac{\omega_{2}}{\omega_{1}}a_{-1,0}(\xi_{0}x^{\frac{\omega_{2}}{\omega_{1}}-1})\frac{\partial}{\partial y}\Bigg{)} + ...
$$

We have now two possibilities:
\begin{itemize}
  \item The slope of the main segment increases and we obtain new weights to continue the resolution of singularities;
  \item The main segment is horizontal.
\end{itemize}

In the first case we must blow up the origin with these new weights. After this new blow-up, once again we have the same two possibilities as before. If it is necessary, we apply $n$ change of coordinates $\Psi_{n}$ given by Lemma \ref{lemma-change-2}. Then or the main segment will be horizontal or we obtain new weights to continue the resolution of singularities. This leads us to a case where the process may never ends and we must apply a change of coordinates of the form
$$\Psi_{\infty}(x,y) = \Big{(}x, y-\sum_{i} \xi_{i}x^{\tilde{\omega}_{i}}\Big{)}.$$

After a change of coordinates $\Psi_{\infty}$, the main segment is horizontal, since at each step the slope increases. Now we must show that such change of coordinates is analytic. Indeed, we remark that the impasse curve $\Delta$ always contains this non elementary point $p$. From the way we defined $\Psi_{\infty}$, it follows that the curve $\{y-\sum_{i} \xi_{i}x^{\tilde{\omega}_{i}} = 0\}$ is contained in the curve $\Delta$. In other words, the curve $\{y-\sum_{i} \xi_{i}x^{\tilde{\omega}_{i}} = 0\}$ must be a component of the analytic set $\Delta$. Since the impasse set is an analytic curve, so is the curve $\{y-\sum_{i} \xi_{i}x^{\tilde{\omega}_{i}} = 0\}$ and therefore the formal sum converges.
\end{proof}

\textbf{Example:} Consider the vector field
\begin{equation}\label{ex-resolution-height-case-1}
 \scriptstyle (y - x^{s})\dot{x} = (y - x^{s})^{r}x; \ (y - x^{s})\dot{y} = -s(y - x^{s})^{r+1} + sy(y - x^{s})^{r} + y^{k+1},
\end{equation}
whose auxiliary vector field is
\begin{equation}\label{ex-resolution-height-case-1-auxiliary}
\scriptstyle X_{A} = (y - x^{s})(y - x^{s})^{r}x\frac{\partial}{\partial x}  + (y - x^{s})\Big{(}-sy^{-1}(y - x^{s})^{r+1} + s(y - x^{s})^{r} + y^{k}\Big{)}y\frac{\partial}{\partial y};
\end{equation}
where $k\geq r+1$. Observe that the origin is an isolated equilibrium point of the adjoint vector field. The main vertex is of the form $(0,r)$ and therefore the height is $r$. Blowing up the auxiliary vector field in origin with weight $\omega = (1,s)$ in the $x$ direction, we obtain
$$
\scriptstyle \widetilde{X}_{A} = (\tilde{y} - 1)(\tilde{y} - 1)^{r}x\frac{\partial}{\partial x} + (\tilde{y} - 1)\Big{(}-s\tilde{y}^{-1}(\tilde{y} - 1)^{r+1} + \tilde{x}^{(k-r)s}\tilde{y}^{k} \Big{)}\tilde{y}\frac{\partial}{\partial y}.
$$

Observe that at the origin the height has decreased, but when we translate origin to $(0,1)\in\mathbb{E}_{x}$ we see that the height is still $r$. So, before the blow-up we must apply a change of coordinates of the form
$$\bar{x} = x; \ \bar{y} = y - x^{s},$$
and then \eqref{ex-resolution-height-case-1-auxiliary} takes form
\begin{equation}\label{ex-resolution-height-case-1-auxiliary-change}
\scriptstyle \overline{X}_{A} = \bar{y}\bar{y}^{r}\bar{x}\frac{\partial}{\partial \bar{x}}  + \bar{y}(\bar{y} + \bar{x}^{s})^{k}\frac{\partial}{\partial \bar{y}},
\end{equation}
and now we have a new weight vector to continue the process.

\section{Global resolution of singularities}\label{sec-global-resolution}
\noindent

We have already studied the local strategy of resolution of singularities and now we are able to give the proof of Theorem \ref{teo-resolution-singularities}. Suppose that $(\mathcal{M}, \mathcal{X},\mathcal{I})$ is a real analytic constrained differential system defined on a 2-dimensional compact manifold $\mathcal{M}$ with corners, which means that $\mathcal{X}$ is a 1-dimensional analytic foliation and $\Delta$ is an 1-dimensional analytic subset of $\mathcal{M}$.

Recall that, in our firsts assumptions, in the local representations of $\mathcal{X}$, the components $P$ and $Q$ of $X$ do not have common factor, which implies that all equilibrium points of the adjoint vector field are isolated. Analogously, the impasse ideal $\mathcal{I}$ is locally generated by an irreducible analytic function $\delta$, and then $\delta$ has isolated singular points. Since all of our objects are analytic, there is a finite number of tangency points between the foliation and the impasse set. Due to the compactness of $\mathcal{M}$, there is finite set $\{p_{i}\}_{i = 1}^{k} $ of non elementary points.

In each point $p_{i}$ we fix local coordinates $(\phi_{i}, U_{i})$ such that the Newton polygon of the auxiliary vector field is controllable. Applying Theorem \ref{teo-local-strategy} in each system of coordinates, we have $(\tilde{\phi}_{i}, \widetilde{U}_{i})$ such that in each $\widetilde{U}_{i}$ there are only elementary points. Since the blow ups are analytic diffeomorphisms that preserve orientation outside the exceptional divisor (as well as the admissible change of coordinates), then the constrained system remains well defined in the intersection between two charts.

Then we obtain a finite sequence of weighted blow-ups
$$(\widetilde{\mathcal{M}}, \widetilde{\mathcal{X}}, \widetilde{\mathcal{I}}) = (\mathcal{M}_{n}, \mathcal{X}_{n}, \mathcal{I}_{n}) \xrightarrow{\Phi_{n}} \cdots \xrightarrow{\Phi_{0}} (\mathcal{M}_{0}, \mathcal{X}_{0}, \mathcal{I}_{0}) = (\mathcal{M}, \mathcal{X}, \mathcal{I})$$
such that the strict transformed $(\widetilde{\mathcal{M}}, \widetilde{\mathcal{X}}, \widetilde{\mathcal{I}})$ has only elementary points.

\section{Acknowledgments}

The first author is supported by Sao Paulo Research Foundation (FAPESP) (grants 2016/22310-0 and 2018/24692-2), and by Coordena\c c\~ao de Aperfei\c coamento de Pessoal de N\'ivel Superior - Brasil (CAPES) - Finance Code 001. The second author was financed by CAPES-Print and FAPESP. The authors are grateful for Daniel Cantergiani Panazzolo for many discussions and suggestions, and for LMIA-Universit\'e de Haute-Alsace for the hospitality during the preparation of this work.





\end{document}